\documentclass[12pt]{amsart}
\usepackage{hyperref}
\usepackage{color}

\usepackage{amsmath,amsfonts,amssymb}

\usepackage{graphics, pinlabel, color}
\usepackage{comment}

\usepackage[all,graph]{xy}
\usepackage{epstopdf}
\usepackage{tikz}
\usetikzlibrary{intersections}

\theoremstyle{plain}

\newtheorem{theorem}{Theorem}[section]
\newtheorem{lemma}[theorem]{Lemma}
\newtheorem{corollary}[theorem]{Corollary}
\newtheorem{proposition}[theorem]{Proposition}

\newtheorem{question}[theorem]{Question}
\numberwithin{equation}{section}

\theoremstyle{remark}
\newtheorem{remark}[theorem]{Remark}

\newtheorem*{claim}{Claim}

\theoremstyle{definition}

\def\Q{\Bbb{Q}}

\def\Z{\Bbb{Z}}
\def\R{\Bbb{R}}

\def\Bl{\operatorname{Bl}}
\def\CP2{\mathbb{CP}^2}
\def\N{\Bbb{N}}
\def\be{\begin{equation}} \def\ee{\end{equation}}
\def\id{\operatorname{id}}

\def\l{\lambda}
\def\L{\L_{\R}}
\def\ll{\langle}
\def\rr{\rangle}

\def\ext{\operatorname{Ext}}
\def\coker{\operatorname{coker}}
\def\sm{\setminus}
\def\bp{\begin{pmatrix}}
\def\ep{\end{pmatrix}}
\def\bn{\begin{enumerate}}
\def\en{\end{enumerate}}

\def\rk{\operatorname{rank}}
\def\ba{\begin{array}}
\def\ea{\end{array}}

\def\L{\Lambda}

\def\a{\alpha}
\def\b{\beta}

\def\G{\Gamma}
\def\s{\sigma}
\def\t{\theta}

\def\lk{\operatorname{lk}}

\def\fr12{\frac{1}{2}}

\def\ol{\overline}
\def\wt#1{\widetilde{#1}}

\def\S{\Sigma}

\def\sp{\operatorname{sp}}
\def\wsp{\operatorname{wsp}}

\def\ord{\operatorname{ord}}

\def\hom{\operatorname{Hom}}
\def\Ext{\operatorname{Ext}}

\def\cl{\operatorname{cl}}

\def\O{\Omega}

\def\p{\partial}
\def\tor{\operatorname{Tor}}
\def\tordelta{\Delta^{\operatorname{tor}}}

\def\wti{\widetilde}

\def\deltator{\Delta^{\operatorname{tor}}}

\def\toiso{\xrightarrow{\simeq}}
\def\K{\mathcal{K}}
\def\J{\mathcal{J}}

\def\co{\colon}
\def\op{\operatorname}
\def\im{\op{Im}}

\begin{document}

\title{Blanchfield forms and Gordian distance}

\author{Maciej Borodzik}
\address{Institute of Mathematics, University of Warsaw, Warsaw, Poland}
\email{mcboro@mimuw.edu.pl}

\date{\today}
\author{Stefan Friedl}
\address{Fakult\"at f\"ur Mathematik\\ Universit\"at Regensburg\\   Germany}
\email{sfriedl@gmail.com}

\author{Mark Powell}
\address{
  D\'{e}partement de Math\'{e}matiques,
  UQAM, Montr\'{e}al, QC, Canada
}
\email{mark@cirget.ca}

\def\subjclassname{\textup{2010} Mathematics Subject Classification}
\expandafter\let\csname subjclassname@1991\endcsname=\subjclassname
\expandafter\let\csname subjclassname@2000\endcsname=\subjclassname


\begin{abstract}
Given a link in $S^3$ we will use invariants derived from the Alexander module
and the Blanchfield pairing to obtain lower bounds on the Gordian distance between links, the unlinking number and various splitting numbers.
These lower bounds generalise results recently obtained by Kawauchi.

We give an application restricting the knot types which can arise from a sequence of splitting operations on a link. This allows us to answer a question asked by Colin Adams in 1996.
\end{abstract}

\maketitle

\section{Introduction}

\subsection{Lower bounds on the clasp number and the Gordian distance}
In this paper, by an $m$-component link $L\subset S^3$ we mean an embedding
\[ L\co \cup_{i=1}^m S^1\times \{i\} \to S^3.\]
Given $i=1,\dots,m$ we write $L_i=L(S^1\times \{i\})$ and we endow it with the orientation inherited from the standard orientation of $S^1$. By a slight abuse of notation we often denote the union of the $L_i$ again by $L$.
Throughout the paper we will identify  two $m$-component links $L$ and $J$ if there exists an isotopy through links from $L$ to $J$. Slightly more informally, an $m$-component link is an isotopy class of an oriented, ordered collection of $m$ disjoint circles in $S^3$.

Let $L$ and $J$ be  $m$-component links in $S^3$.
 We are interested in the following measures of how different $L$ and $J$ are.
\bn
\item The \emph{Gordian distance $g(L,J)$} which is defined as the minimal number of crossing changes required to turn a diagram representing $L$ into a diagram for  $J$. Here we take the minimum over all diagrams of $L$.
\item The \emph{4-dimensional clasp number $c(L,J)$} which is the minimal number of double points of an immersed concordance between $L$ and $J$. An \emph{immersed
concordance} is a proper immersion of $m$ annuli $f_j \colon S^1 \times I \looparrowright S^3 \times I$ with   $f_j(S^1 \times \{0\}) = L_j\times\{0\}$ and $f_j(S^1 \times \{1\}) = -J_j\times \{1\}$,  for $j=1,\dots,m$. The only allowed singularities of the immersion
are ordinary double points.
\en
Note that $c(L,J) \leq g(L,J)$ since a sequence of crossing changes and isotopies gives rise to an immersed concordance, with one double point for each crossing change.

Our goal in this paper is to give lower bounds on the Gordian distance and the 4-dimensional clasp number from the Alexander module and the Blanchfield pairing of a link. The relationship between the Alexander module and the unlinking number has been explored in several earlier papers, see e.g.\ \cite{BW84,Tr88,CFP13}. The first and second authors undertook a systematic study of the relationship between the Blanchfield pairing and the unknotting number of knots in \cite{BF12,BF13}.

Given an $m$-component  link $L$ we refer to the complement of an open tubular neighbourhood of $L$ as the \emph{exterior of $L$} and we denote it by $X_L=S^3\sm \nu L$.
We write  $\L=\Z[t_1^{\pm 1},\dots,t_m^{\pm 1}]$. We can associate the Alexander module $H_1(X_L;\L)$ to $L$ and we denote  the rank of the Alexander module by $\b(L):=\rk_{\L}(H_1(X_L;\L))$.
The Alexander polynomial  $\Delta_L\in \L$ of $L$  is defined as the order of the Alexander module. Note that $\Delta_L=0$ if and only if $\b(L)>0$. We also consider the \emph{torsion Alexander polynomial} $\deltator_L$ as the order of the torsion submodule $\tor_{\L} H_1(X_L;\L)$. Note that $\deltator_L$ is always non-zero.

Next we consider the Blanchfield form, which was first introduced for knots by Blanchfield \cite{Bl57} in 1957. Let $S \subset \L$ be the multiplicative subset generated by the polynomials $t_i-1$, for $i=1,\dots,m$.  By inverting the elements of $S$ we obtain the ring
\[\L_S := \Z[t_1^{\pm 1},\dots,t_m^{\pm 1},(t_1-1)^{-1},\dots,(t_m-1)^{-1}].\]
Furthermore we denote the quotient field of $\L$ by $\Omega=\Q(t_1,\dots,t_m)$. This is also the quotient field of $\L_S$. The Blanchfield form
\[\Bl_L\co  \hat{t}H_1(X_L;\L_S)\times \hat{t}H_1(X_L;\L_S)\to \Omega/\L_S\]
is a nonsingular, hermitian, sesquilinear form defined on a certain quotient $\hat{t}H_1(X_L;\L_S)$ of the  torsion submodule $\tor_{\L_S} H_1(X_L;\L_S)$;
see Section~\ref{section:blanchfield} for details.
We say that a hermitian form $\lambda\co H\times H\to \Omega/\L_S$ is \emph{metabolic} if it admits a \emph{metabolizer}, i.e.\ a submodule $P$ of $H$ with $P=P^\perp:=\{h\in H\,|\, \lambda(p,h)=0\mbox{ for all }p\in P\}$.
The following is our first main theorem.

\begin{theorem}\label{mainthm}
Let $L$ and $J$ be  $m$-component links.  Then  $|\beta(L)-\beta(J)| \leq c(L,J)$.
Moreover, if $c(L,J) = |\beta(L) - \beta(J)|$, then the Witt sum of Blanchfield forms $\Bl_L \oplus -\Bl_{J}$ is metabolic.
\end{theorem}

In the following, given $f=f(t_1,\dots,t_m)\in \L$ we write
$\ol{f}:=f(t_1^{-1},\dots,t_m^{-1})$.
Furthermore, we say that a polynomial $n\in \L$ is \emph{negligible}
if it is of the form
\[ n=\pm  \prod_{i=1}^m  t_i^{r_i}\,\,\cdot \,\prod_{i=1}^m  (1-t_i)^{s_i}\]
where $r_i,s_i\in \Z$ for  $i=1,\dots,m$; this is equivalent to saying that $n$ is invertible in $\L_S$.
We can formulate the following straightforward corollary to Theorem \ref{mainthm}.

\begin{corollary}\label{maincor}
Let $L$ and $J$ be  $m$-component links.
If $c(L,J) = |\beta(L) - \beta(J)|$, then $\deltator_L\cdot f\, \ol{f}= \deltator_{J}\cdot g\, \ol{g}\cdot n$ for some non-zero $f,g\in \L$ and some negligible $n\in \L$.
\end{corollary}

The inequality  $|\beta(L)-\beta(J)| \leq c(L,J)$ in Theorem \ref{mainthm} and the statement of the corollary are essentially the main result of a recent paper  by Kawauchi \cite{Ka13}. (Kawauchi gives a slightly more precise version of the corollary in so far as he also determines the negligible element $n$.) In the case of single variable Alexander modules, the rank estimate was previously given by Kawauchi in \cite{Ka96}.
The result on Blanchfield forms is to the best of our knowledge a new result.

Our second main theorem gives a refinement of Corollary \ref{maincor} when we replace the clasp number by the Gordian distance. More precisely we have the following theorem.

\begin{theorem}\label{lem:torsionpolynomials}
Let $L$ and $J$ be two $m$-component links.
Then
\[ |\b(L)-\b(J)|\leq g(L,J).\]
Furthermore, if  $\b(J)=\b(L)+g(L,J)$, then
\[  \deltator_L=\deltator_{J}\cdot f\, \ol{f}\cdot n\]
for some $f\in \L$ and some negligible  $n\in \L$.
\end{theorem}

Put differently, since Gordian distance is more specialized than 4-dimensional clasp number, we are able to show that one torsion polynomial divides the other.

\subsection{The splitting number and the weak splitting number}

We now recall and introduce a few more link theoretic notions.
\bn
\item The \emph{unlinking number $u(L)$} of an $m$-component link $L$ is the Gordian distance to the $m$-component unlink.
\item An $m$-component link $L$ is  a \emph{split link} if there are $m$ disjoint balls  in $S^3$ each of which contains a component of $L$.
\item  Following \cite{BS13,CFP13}, the \emph{splitting number $\sp(L)$} of a link~$L$ is the minimal number of crossing changes between \emph{different} components of $L$ required on some diagram of $L$ to obtain a split link, where the minimum is taken over all diagrams.
\item The \emph{weak splitting number} $\wsp(L)$ of a link $L$ is the minimal number of crossing changes required on some diagram of $L$ to produce a split link, where the minimum is taken over all diagrams.
\en
Note that for the weak splitting number, unlike the splitting number considered above, crossing changes of a component with itself are permitted.  For example if~$L$ is the Whitehead link then it is straightforward to see that $\wsp(L)=1$, but an elementary linking number argument (see \cite[Section~2]{CFP13}) shows that $\sp(L)=2$.
Somewhat confusingly the weak splitting number is referred to as the splitting number in \cite{Ad96,Sh12,La14}, but we decided to follow the convention used
by Batson--Seed \cite{BS13}.

It is straightforward to apply Theorem \ref{lem:torsionpolynomials} to the computation of unlinking numbers, splitting numbers and weak splitting numbers.
The precise statements are given in Corollaries \ref{cor:unlink}, \ref{cor:split} and \ref{cor:wsp}. We note that the result on splitting numbers, Corollary \ref{cor:split}, considerably strengthens \cite[Theorem~4.2]{CFP13}.

In Section~\ref{section:examples-unlinking-splitting} we give some examples of the use of these corollaries.
Corollary~\ref{cor:unlink} enables us to easily compute the unlinking numbers of the 3-component links with 9 or fewer crossings.  We also show that some, but not all, of the results on splitting numbers from~\cite{CFP13} which were obtained using covering links, can also be obtained with the algebraic methods of this paper.

\subsection{Knot types from weak splitting operations}
Now we turn to an application of Theorems~\ref{mainthm} and \ref{lem:torsionpolynomials} to weak splitting numbers.
Let us introduce some notation.
If a link $J$ can be obtained from a link $L$ by a sequence of $r$ crossing changes then we write $L \rightsquigarrow_{r} J$.
A sequence of crossing changes culminating in a split link is referred to as a \emph{splitting sequence}.
 Given knots $K_1,\dots,K_m$ we denote the split link comprising these knots by $K_1 \sqcup \dots \sqcup K_m$.
We write $U$ for the unknot throughout this subsection.

In Section~\ref{section:knot-types-weak-splitting} we give a general condition in terms of Blanchfield forms and Alexander polynomials restricting the knot types which can arise from a sequence of crossing changes realising the weak splitting number; see Theorem~\ref{thm:weak-split-no-blanchfield-r-plus-one-component}.

In~\cite{Ad96} Adams gave some examples of 2-component links~$L$ with unknotted components and $\wsp(L)=1$, such that whenever one turns $L$ into a split link using a single crossing change, the resulting split link has a knotted component.  Put differently, for the given link $L$ one has to pay a price for splitting it with one crossing change, i.e.\ one has to turn one of the unknotted components into a non-trivial knot.

Adams then asked whether there are occasions when the price to pay must be `arbitrarily high'.
More precisely, the following question was asked by Adams~\cite[p.~299]{Ad96}.

\begin{question}\label{question:adams}
Let $C$ be a complexity for knots, e.g.\ crossing number, hyperbolic volume, span of some knot polynomial. Given any $c\in \N$, does there exist a $2$-component link $L$ with unknotted components such that for any splitting sequence $L\rightsquigarrow_1 K\sqcup U$ of length one we have $C(K)\geq c$?
\end{question}

We give an affirmative answer to Adams' question for the crossing number.

\begin{theorem}\label{thm:high-crossing-no-knots-from-wsp-one-links}
Fix $c \in \mathbb{N}$.  There exists a 2-component link $L$ with unknotted components such that such for any knot $K$ with $L \rightsquigarrow_1 K \sqcup U$, the crossing number of $K$ is at least $c$.
\end{theorem}

Next we give a quick summary of the proof of Theorem~\ref{thm:high-crossing-no-knots-from-wsp-one-links}.
Given $c\in \N$ we combine constructions from \cite{Ad96} and \cite{Kon79}
to obtain a link $L$ with  $L \rightsquigarrow_1 J \sqcup U$ where $J$ is a knot such that the degree of $\Delta_J(t)$ is $2n$ and $\Delta_J$ is irreducible, chosen so that $n$ is suitably high with respect to $c$.  Then we find that for any $K$ as in Theorem~\ref{thm:high-crossing-no-knots-from-wsp-one-links}, we have $\Delta_J | \Delta_K$, so that the degree of $\Delta_K$ is forced to be at least $2n$.  The theorem follows since the degree of the Alexander polynomial gives a lower bound  on the crossing number.

In fact, it is straightforward to modify the proof  of Theorem \ref{thm:high-crossing-no-knots-from-wsp-one-links}
to give an affirmative answer to Adams' question for the  support of knot Floer homology and the 3-genus as the complexity, since the Alexander polynomial provides a lower bound for these as well.
\medskip

The paper is organized as follows.
In Section \ref{section:blanchfield} we recall the definitions and basic properties of the Alexander module and the Blanchfield form of a link.
In Section \ref{section:clasp} we provide the proof of Theorem \ref{mainthm} and Corollary \ref{maincor}. In Section \ref{section:gordian} we will prove Theorem \ref{lem:torsionpolynomials} and we will state several corollaries relating Alexander polynomials to the unlinking number, the splitting number and the weak splitting number of a link.
In Section~\ref{section:examples-unlinking-splitting} we give examples of unlinking and splitting number computations.
In Section~\ref{section:knot-types-weak-splitting} we investigate weak splitting numbers and the knot types arising from them; in particular we give the proof of Theorem~\ref{thm:high-crossing-no-knots-from-wsp-one-links}.

\subsection*{Conventions.} All rings are commutative and all modules are finitely generated.
Links are oriented and ordered.

\subsection*{Acknowledgment.}
Work on this paper was supported by the SFB 1085 `Higher Invariants' at the Universit\"at Regensburg funded by the DFG.
This paper has roots in joint work of two of us with Jae Choon Cha, and we wish to thank him, and also Patrick Orson, for many helpful conversations. We
would also like to thank Lorenzo Traldi for helpful comments on the first version of the article.

\section{The Alexander module and the Blanchfield form}
\label{section:blanchfield}

\subsection{Alexander modules}
Throughout the paper we identify the group ring of $\Z^m$ with the multivariable Laurent polynomial ring $\L:=\Z[t_1^{\pm 1},\dots, t_m^{\pm 1}]$ in the canonical way. (We suppress $m$ from the notation, but it will always be clear from the context which $m$ we mean.)
We denote  by $f\mapsto \ol{f}$ the involution on $\L$ which is given by the unique $\Z$-linear ring homomorphism with $t_i\mapsto t_i^{-1}$, $i=1,\dots,m$.
Furthermore, given a $\L$-module $V$ we denote  the $\L$-module with involuted $\L$-module structure by $\ol{V}$, i.e.\ the underlying additive group of $\ol{V}$ is the same as for $V$, but the action of $f$ on $\ol{V}$ is defined as the action of $\ol{f}$ on $V$.

Throughout the paper we will mostly be interested in the following $\L$-modules:
\bn
\item the ring $\L$ itself,
\item the ring $\L_S := \Z[t_1^{\pm 1},\dots, t_m^{\pm 1},(t_1-1)^{-1},\dots, (t_m-1)^{-1}]$, which  is  the multivariable Laurent polynomial ring with the monomials $t_i-1$ inverted,
\item the quotient field   $\Omega$ of $\L$, which is also the quotient field of $\L_S$.
\en

In the following let $X$ be a connected manifold and let $\varphi\co \pi_1(X)\to \Z^m$ be a homomorphism.
We denote the cover corresponding to $\varphi$ by $p\co \wti{X}\to X$. Given $Y\subset X$ we write $\wti{Y}:=p^{-1}(Y)$. Note that the group $\Z^m$ acts by deck transformations on $C_*(\wti{X},\wti{Y})$ on the left. Thus we can view $C_*(\wti{X},\wti{Y})$ as a (left) $\Z[\Z^m]=\L$-module.
Now let $M$ be a (left) $\L$-module. Then we consider the following  (right) $\L$-modules:
\[ \ba{rcl} H_i(X,Y;M)&:=&H_i\left(\ol{C_*(\wti{X},\wti{Y})}\otimes_{\L}M\right),\mbox{ and }\\[2mm]
H^i(X,Y;M)&:=&H_i\left(\hom_{\L}(C_*(\wti{X},\wti{Y}),M)\right).\ea \]
As usual we write $ H_i(X;M)= H_i(X,\emptyset;M)$ and $ H^i(X;M)= H^i(X,\emptyset;M)$.

Now let $L$ be an $m$-component  link.
We denote the exterior of $L$ by  $X_L$.
Note that $\pi_1(X_L)$ admits a canonical epimorphism $\phi_L$ onto $\Z^m$ which is given by sending the $i$-th oriented meridian to the $i$-th vector of the standard basis of $\Z^m$.  In the following we will refer to $H_1(X_L;\L)$ as the  \emph{Alexander module of $L$}.

We recall several basic properties of twisted homology and cohomology groups.
The following lemma is well-known; see e.g.\ \cite[Section~VI.3]{HS71}.

\begin{lemma} \label{lem:h0}
Let  $X$ be a connected CW-complex and let $\varphi\co \pi_1(X)\to \Z^m$ be a homomorphism. Then for any $\L$-module $M$ we have
\[  H_0(X;M)\cong  M/\Big\{  \sum_{i=1}^m \varphi(g_i)v_i-v_i\,|\,v_1,\dots,v_m\in M\mbox{ and } g_1,\dots,g_m\in\pi_1(X)\Big\}.\]
In particular, if $M$ is a field and $\varphi$ is non-trivial, then $H_0(X;M)=0$.
\end{lemma}

The following theorem is a well-known instance of Poincar\'e duality and the Universal Coefficient Theorem.

\begin{theorem} \label{thm:pd}
Let  $X$ be a connected oriented $k$-manifold and let $\varphi\co \pi_1(X)\to \Z^m$ be a homomorphism. Let $\partial X=Y_0\cup Y_1$ be a decomposition of the boundary into two submanifolds with $\partial Y_0=\partial Y_1$.  Then for any $\L$-module $M$ we have an isomorphism
\[  \op{PD}\co H_i(X,\partial Y_0;M)\xrightarrow{\cong}  \ol{H^{k-i}(X,\partial Y_1;M)}.\]
In particular, if $M=\Omega$ is the quotient field of $\L$, then
$H_i(X,\partial Y_0;\Omega)\cong  H_{k-i}(X,\partial Y_1;\Omega)$.
\end{theorem}

\subsection{Ranks and orders of modules}
Let $R$ be a domain with quotient field $Q$. Let $M$ be an $R$-module.
We then refer to $\rk_R(M):=\dim_Q(M\otimes_R Q)$ as the \emph{rank} of $M$.
Now suppose that $R$ is in fact a UFD and that $M$ is finitely generated.
We pick a resolution
\[ R^k\xrightarrow{A} R^{\,l}\to M\to 0\]
with $k\in \N\cup \{\infty\}$. After adding possibly columns of zeros we can and will assume that $k\geq l$. The \emph{order $\ord(A)$ of $M$} is defined as the greatest common divisor of the $l\times l$-minors of $A$. Note that the order is well-defined up to multiplication by a unit in $R$; see \cite{CF77} for details.
Also note that $\ord(A)\ne 0$ if and only if $\rk(A)=0$.

For future reference we record the following lemma.  A proof can be found in~\cite[Chapter~3.3]{Hi12}.

\begin{lemma}\label{lem:ordermultiplicative}
Let $R$ be a UFD and let
\[ 0\to A\to B\to C\to 0 \]
be a short exact sequence of finitely generated $R$-modules. Then
\[ \ord(B)=\ord(A)\cdot \ord(C).\]
\end{lemma}

In the following, given $f,g\in R$ we write $f\doteq g$ if $f=ug$ for some unit $u\in R$. We will mostly be interested in the rings $R=\L$ and $R=\L_S$.
Note that the units in $\L=\Z[t_1^{\pm 1},\dots, t_m^{\pm 1}]$ are precisely the monomials $\pm t_1^{n_1}\dots t_m^{n_m}$. Furthermore, the units in
\[\L_S = \Z[t_1^{\pm 1},\dots, t_m^{\pm 1},(t_1-1)^{-1},\dots, (t_m-1)^{-1}]\]
are precisely the products of monomials and powers of $t_i-1$, $i=1,\dots,m$.
Put differently, the  units in  $\L_S$   are the negligible elements from the introduction.

Given an $m$-component  link $L$ we write
\[\b(L):=\rk_{\L}(H_1(X_L;\L)).\]
Note that $X_L$ is compact, so in particular $X_L$ is homeomorphic to a finite CW-complex, which in turn implies that the cellular chain complex $C_*(\wti{X_L})$ is finitely generated over~$\L$. Since $\L$ is Noetherian it follows that $H_1(X_L;\L)$ and $\tor_{\L} H_1(X_L;\L)$ are finitely generated $\L$-modules.
We refer to
\[\Delta_L:=\ord(H_1(X_L;\L))\]
as the \emph{Alexander polynomial of $L$}. Furthermore, we refer to
\[\deltator_L:=\ord(\tor_{\L} H_1(X_L;\L))\]
as the \emph{torsion Alexander polynomial of~$L$}.

Recall that an $m$-component link is \emph{split} if there exist $m$ disjoint 3-balls in $S^3$, each of which contains a component of $L$.
Later on we will need the following well-known lemma.

\begin{lemma} \label{lem:alexsplit}
Let $L$ be a split $m$-component link. Then $\b(L)=m-1$ and $\deltator_L\doteq \Delta_{L_1}(t_1)\cdot \ldots \cdot \Delta_{L_m}(t_m)$.
\end{lemma}

\begin{proof}
We just provide a short sketch of the well-known proof.
By our hypothesis there exist disjoint balls $B_1,\dots,B_m$ in $S^3$, such that
$L_i\subset B_i$, $i=1,\dots,m$.  We write $S_i=\partial B_i$, $i=1,\dots,m$ and $C:=\ol{S^3\sm \cup_{i=1}^m B_i}$.

Note that $H_1(S_i;\L)=H_1(C;\L)=0$ and  $H_0(S_i;\L)=H_0(C;\L)=\L$.
Also, a straightforward argument shows that for $j=0,1$ and $i\in \{1,\dots,m\}$ we have
\[ H_j(B_i\sm \nu L_i)\cong H_j(S^3\sm \nu L_i;\L)\cong H_j(S^3\sm \nu L_i;\Z[t_i^{\pm 1}])\otimes_{\Z[t_i^{\pm 1}]}\L.\]
It follows easily from Lemma \ref{lem:h0} that $H_0(S^3\sm \nu L_i;\Z[t_i^{\pm 1}])\cong
\Z[t_i^{\pm 1}]/(t_i-1)\Z[t_i^{\pm 1}]$. Furthermore it follows from the definitions that $M_i:=H_1(S^3\sm \nu L_i;\Z[t_i^{\pm 1}])$ is a module whose order equals $\Delta_{L_i}(t_i)$.

The  Mayer--Vietoris sequence with $\L$ coefficients corresponding to the decomposition $S^3\sm \nu L=\bigcup_{i=1}^m (B_i\sm \nu L_i)\,\cup \, C$ gives rise to the following exact sequence:
\[\xymatrix@R-0.8cm @C-0.1cm{
0\ar[r]&\bigoplus_{i=1}^m M_i\otimes_{\Z[t_i^{\pm 1}]}\L\ar[r]& H_1(S^3\sm \nu L;\L)\ar[r]& \L^m\ar[r]&\\
\ar[r]& \L\oplus \bigoplus_{i=1}^m \L/(t_i-1)\L\ar[r]& H_0(S^3\sm \nu L;\L)\ar[r]& 0.
}\]
By Lemma \ref{lem:h0} the module  $H_0(S^3\sm \nu L;\L)$ is $\L$-torsion.
The lemma  now follows immediately from the above observations and from elementary properties of ranks and orders. We leave the details to the reader.
\end{proof}

\subsection{The maximal pseudo-null submodule}\label{section:pseudo-null-submodule}
Given a ring $R$ and a  module $M$ over $R$ we denote the torsion submodule by $TM$.
Furthermore we denote the \emph{maximal pseudo-null submodule of $M$} by  $zM$; this is the submodule of $TM$
which is generated by the elements of $M$  whose annihilator is not contained in any principal ideal of $R$.  Following \cite[Section~2.3]{Hi12} we write
\[\hat{t}M := TM/zM.\]

For future reference we record the following lemma, see  \cite[Theorem~3.5]{Hi12}.

\begin{lemma}\label{lem:sameorder}
For any $\L$-module $M$ we have $ \ord(TM)=\ord(\hat{t}M)$.
\end{lemma}

\subsection{Linking forms}\label{section:linking-forms}
Let $R$ be a ring with (possibly trivial) involution. We denote the quotient field of $R$ by $Q$. Here and throughout the paper we extend the involution on $R$ to an involution on $Q$.
Let  $\lambda\co H\times H\to Q/R$ be a map.
\begin{itemize}
\item We say $\l$ is \emph{sesquilinear} if $\l(au+bv,w)=\ol{a}\l(u,w)+\ol{b}\l(v,w)$ and
$\l(u,av+bw)=\l(u,v)a+\l(u,w)b$  for all $a,b\in R$ and $u,v,w\in H$.
\item We say $\l$ is \emph{hermitian} if $\lambda(u,v)=\ol{\lambda(v,u)}$ for all  $u,v\in H$.
\item We  say that $\l$ is \emph{nonsingular} if the assignment $u\mapsto (v\mapsto \l(u,v))$ defines an isomorphism of $R$-modules $\ol{H}\to \hom(H,Q/R)$.
\item A \emph{linking form over $R$} is an $R$-module $H$ together with a hermitian sesquilinear nonsingular form $\lambda\co H\times H\to Q/R$.
\item We say that two linking forms $\lambda \colon H \times H \to Q/R$ and $\mu \colon G \times G \to Q/R$ are \emph{isomorphic} if there is an isomorphism of $R$ modules $f \colon H \toiso G$ such that $\lambda(v,w) = \mu(f(v),f(w))$ for all $v,w\in H$.
\item We say that the linking form \emph{$\l$ is metabolic} if $\l$ admits a \emph{metabolizer}, i.e.\ a submodule $P$ of $H$ with $P=P^\perp:=\{v\in H\,|\, \lambda(p,v)=0\mbox{ for all }p\in P\}$.
\item Given a linking form  $\lambda\co H\times H\to Q/R$  we write
$-\lambda\co H\times H\to Q/R$ for the linking form which is defined by $(-\l)(v,w)=-\l(v,w)$ for all $v,w\in H$.
\item Given two linking forms  $\lambda\co H\times H\to Q/R$ and  $\lambda'\co H'\times H'\to Q/R$
we refer to
\[ \ba{rcl} \l\oplus \l'\co (H\oplus H')\times (H\oplus H') &\to & Q/R \\
((v\oplus v'),(w\oplus w'))&\mapsto & \l(v,w)+\l'(v',w') \ea \]
as the \emph{Witt sum of $\l$ and $\l'$}.
\item We say that two linking forms $(H,\lambda)$ and $(G,\mu)$ are \emph{Witt equivalent},
written as $(H,\lambda) \sim (G,\mu)$,  if there exist metabolic forms $(P,\varphi)$ and $(Q,\phi)$ such that
$$(H,\lambda) \oplus (P,\varphi) \cong (G,\mu) \oplus (Q,\phi).$$
\item The \emph{Witt group of linking forms over $R$} is defined as the set of Witt equivalence  classes of linking forms.
The identity element of the Witt group is the equivalence class of the linking form on the zero module and the inverse of $[(H,\lambda)]$ is given by $[(H,-\lambda)]$.
\end{itemize}

For the record we state the following well-known lemma which is proved in \cite[Lemma~3.26]{Hi12}, for example.

\begin{lemma}\label{lem:ordernorm}
Let $R$ be a UFD with (possibly trivial) involution.
If $\lambda\co H\times H\to Q/R$ is a metabolic linking form,   then
$\ord(H) \doteq f \cdot  \ol{f}$ for some $f\in R$.
\end{lemma}

\subsection{The Blanchfield form}
In this section we sketch the definition of Blanchfield forms for 3-manifolds and we summarize a few key properties.  We refer to \cite[Section~2]{Hi12} for a thorough treatment of Blanchfield forms.

Let $N$ be a $3$-manifold (with possibly nonempty boundary) and let $\varphi\co \pi_1(N)\to \Z^m$ be a homomorphism, which induces a homomorphism $\Z[\pi_1(N)] \to \Z[\Z^m] = \L$.  Recall that $\Omega$ is the quotient field of $\L$.  Let $R$ be a subring of $\Omega$ that contains $\L$
and that is closed under involution.
We denote the composition of the following sequence of $R$-module homomorphisms by $\Phi$:
\[\ba{rcl} \ol{TH_1(N;R)}\xrightarrow{(1)} \ol{TH_1(N,\partial N;R)}&\xrightarrow{(2)}& TH^2(N;R)\\
&\xrightarrow{(3)}& \ext^1_R(H_1(N;R),R)\\
&\xrightarrow{(4)}& \ext^1_R(TH_1(N;R),R)\\
&\xrightarrow{(5)}& \hom_R(TH_1(N;R),\Omega/R).\ea\]
Here we used the following maps:
\bn
\item the inclusion induced map;
\item the Poincar\'e duality given in Theorem \ref{thm:pd};
\item  the Universal Coefficient Spectral Sequence (see \cite[Section~2.1~and~2.4]{Hi12}) gives rise to a pair of exact sequences as in the following diagram, where $V$ is some $R$-module.
\[\xymatrix @R-0.3cm @C-0.3cm {&0 \ar[d]&& \\
&\coker\left(\Ext^0_R(H_1(N,R),R) \to \Ext^2_R(H_0(N,R),R)\right)\hskip -3.5cm\ar[d] & & \\
 0 \ar[r] & V \ar[r] \ar[d] & H^2(N;R)\ar[r] & \Ext^0_R(H_2(N,R),R) \\
 & \ker\left(\Ext^1_R(H_1(N,R),R) \to \Ext^3_R(H_0(N,R),R)\right)\hskip -3cm\ar[d] & & \\
&0&& }\]
Since $\Ext^0_R(H_2(N,R),R) = \hom_R(H_2(N,R),R)$ is torsion-free we obtain a map $TH^2(N;R)\to V$, which we then compose with the map
\[V \to \ker\left(\Ext^1_R(H_1(N,R),R)\to \Ext^3_R(H_0(N,R),R)\right)\,  \hookrightarrow \,\Ext^1_R(H_1(N,R),R).\]
\item The map induced by the restriction from $H_1(N;R)$ to $TH_1(N;R)$.
\item The Bockstein long exact sequence arising from the short exact sequence of coefficients $0 \to R \to \Omega \to \Omega/R \to 0$:
\[\xymatrix @R-0.8cm @C-0.2cm {
\ar[r]&\ext^0_{R}(TH_1(N;R),\Omega) \ar[r] &\ext^0_{R}(TH_1(N;R),\Omega/R)\ar[r]&\\
\ar[r]& \ext^1_{R}(TH_1(N;R),R)\ar[r] &\ext^1_{R}(TH_1(N;R),\Omega)\ar[r]&}\]
has first and last terms vanishing, the first since $TH_1(N;R)$ is $R$-torsion and the last since $\Omega$ is an injective $R$-module.  Thus there is a canonical  map $\ext^1_R(TH_1(N;R),R)\to \ext^0_{R}(TH_1(N;R),\Omega/R) = \hom_R(TH_1(N;R),\Omega/R)$.
\en
Hillman \cite[Chapter~2]{Hi12} shows that if $H_i(\partial N;R)=0$ for $i=0,1$ and $R$ contains $\L_S$,
then the resulting map
\[ \ba{rcl}  TH_1(N;R)\times TH_1(N;R)&\to &\Omega/R\\
(x,y)&\mapsto &\Phi(x)(y)\ea \]
descends to a linking form
\[ \Bl_N\co \hat{t}H_1(N;R)\times \hat{t}H_1(N;R)\to \Omega/R\]
that we refer to as the \emph{$R$-Blanchfield form of $N$}.

Later on we will make use of the following proposition that is proved on \cite[p.~40]{Hi12}. (See also \cite[Proposition~2.8]{Le00}.)

\begin{proposition}\label{prop:hillman-metabolic}
Let $N$ be a closed $3$-manifold and let $\varphi\co \pi_1(N)\to \Z^m$ be a homomorphism. Let $R$ be a subring of $\Omega$ which contains $\L_S$.
 Suppose that there exists a $4$-manifold $W$ with $\partial W=N$ such that $\varphi$ extends over $\pi_1(W)$. We write
\[ P:=\im\left\{ TH_2(W,N;R)\xrightarrow{\partial} TH_1(N;R)\to \hat{t}H_1(N;R)\right\}.\]
If the sequence
\[  TH_2(W,N;R)\xrightarrow{\partial} TH_1(N;R)\to  TH_1(W;R)\]
is exact, then $P^\perp$ is a metabolizer for the $R$-Blanchfield pairing of $N$.
\end{proposition}

The proof of this is contained in the proof of \cite[Theorem~2.4]{Hi12}.  However the situation of his Theorem 2.4 is different to ours, in that the 4-manifold $Z$ in \cite{Hi12} is the exterior of a concordance between two links.  Nevertheless the part of his proof on page 40, verbatim except for $(Z,\partial Z)$ replaced with $(W,N)$, provides a proof of Proposition~\ref{prop:hillman-metabolic}.  There was a slight problem with this part of the proof in the first edition of Hillman's book, therefore the reader is advised to consult the second edition.  A more detailed version of the proof is also given in~\cite[Section~5.1]{Kim14}.
\smallskip

Now let $L$ be an $m$-component link. It is straightforward to see that $H_*(\partial X_L;\L_S)=0$. We then refer to
\[ \Bl_L:=\Bl_{X_L} \colon \hat{t}H_1(X_L;\L_S) \times \hat{t}H_1(X_L;\L_S) \to \Omega/\L_S\]
as the \emph{Blanchfield form of $L$}.
Given an $m$-component link $L$ we denote by $\Bl_{L_i}(t_i)$ the linking form  which is given by tensoring the
Blanchfield form of the knot $L_i$ over the ring $\Z[t_i^{\pm 1}]$ up to the ring $\L_S$.
Now we can formulate the following well-known lemma, which can be viewed as a generalization of  Lemma~\ref{lem:alexsplit}.

\begin{lemma} \label{lem:blanchfield-split}
Let $L$ be a split $m$-component link. Then $\b(L)=m-1$ and $\Bl_L\sim \Bl_{L_1}(t_1)\oplus \ldots \oplus \Bl_{L_m}(t_m)$.
\end{lemma}

\begin{proof}
We only provide a sketch of the proof.
The proof of Lemma~\ref{lem:alexsplit} shows that the torsion part of $H_1(X_L;\L)$ is the direct sum of the Alexander modules of the components, tensored up into $\L$.  Each component lives in a 3-ball.  The lemma follows from the observation that the Blanchfield form of a 3-manifold $N$ is isomorphic to that of $N$ with a 3-ball removed. We leave the details to the reader. \end{proof}

\subsection{Brief review of Reidemeister torsion}\label{section:review-reidemeister-torsion}

In this section we remind the reader of the definition and some basic properties of Reidemeister torsion, which we shall apply later to compute $\L$ coefficient homology.
For a comprehensive and readable introduction the reader is referred to~\cite{Tu01}.

Let $(C_*,\{c_i\})$ be a based finite chain complex of finitely generated free $\Omega$-modules. If $C_*$ is not acyclic, then we define $\tau(C_*,\{c_i\})=0$.
Otherwise we pick a basis $b_i$  of each $B_i := \im(\partial_i \colon C_{i+1} \to C_i)$, and we pick a lift $\wt{b}_{i-1}$  of $b_{i-1}$ to $C_i$. The \emph{Reidemeister torsion} of $(C_*,\{c_i\})$ is defined as
\[\tau(C_*,\{c_i\}) := \prod \, \det([b_i\wt{b}_{i-1}/c_i])^{(-1)^{i+1}} \in \Omega^{\times} = \Omega\sm\{0\},\]
where $[d/e]$ is the change of basis matrix between bases $d$ and $e$.  The torsion is independent of the choices of $b_i$ and of the lifts $\wt{b}_{i-1}$.

Let $X$ be a CW complex together with a homomorphism $\varphi \colon \pi_1(X) \to \Z^m$.
 Such a representation induces a homomorphism of rings $$\Z[\pi_1(X)] \to \L \to \Omega.$$
Choose an orientation of each cell and choose a lift of each cell to a cell in the  cover $\wti{X}$ corresponding to $\varphi$.  This determines a basis
$\{c_i\}$ for the chain complex
\[C_* (X;\Omega) = \ol{C_*(\wti{X})} \otimes_{\L} \Omega.\] By Chapman's theorem \cite{Ch74}, the torsion:
\[\tau(X) := \tau(C_*^\a(X;\Omega),\{c_i\})\,\, \in\,\, \{0\}\,\,\cup \,\,\Omega^\times/ \pm t_1^{k_1}\dots t_m^{k_m} \]
is a well-defined homeomorphism invariant of $(X,\varphi)$, that is up to the indeterminacy indicated it is independent of the choice of cellular decomposition, the choice of orientations and the choice of lifts.

An important property of the torsion is multiplicativity in short exact sequences.

\begin{theorem}\label{thm:mult_of_torsions_in_ses}
  Let \[0 \to C \to D \to E \to 0\] be a short exact sequence of finite
 acyclic chain complexes of finitely generated $\Omega$-modules.  Let $\{c_i\}$ and $\{e_i\}$ be bases for $C$ and $E$ respectively, let $\{\wt{e}_i\}$ be a lift of $\{e_i\}$ to $D$, and define the basis $d_i = c_i\wt{e}_i$.  Then $$\tau(D,\{d_i\}) = \pm \tau(C,\{c_i\})\tau(E,\{e_i\}).$$
\end{theorem}
\begin{proof}
See \cite[Theorem~1.5]{Tu01}.
\end{proof}

We will also need the following useful formula.

\begin{lemma}\label{lemma:torsion-X-times-S1}
  Let $X$ be any CW complex, and let $\a \colon \pi_1(X \times S^1) \to \Z^m$ be a homomorphism such that  the composition $ \pi_1(S^1) \to \pi_1(X \times S^1) \xrightarrow{\a} \Z^m\to \Z[\Z^m]=\Lambda$ sends
   a generator of $\pi_1(S^1)$ to a non-trivial element $z\ne 1\in \Lambda$.
  Then $$\tau(X \times S^1) = (z-1)^{-\chi(X)}.$$
\end{lemma}

\begin{proof}
  See \cite[Example~2.7]{Ni03}.
\end{proof}

In particular this lemma implies that the torsion of a torus is $\tau(S^1 \times S^1) = 1$, provided at least one of the generators maps nontrivially to~$\Z^m$.

\section{Blanchfield forms and the 4-dimensional clasp number}
\label{section:clasp}
For the reader's convenience we recall Theorem \ref{mainthm}.\\

\noindent \textbf{Theorem \ref{mainthm}.}\emph{
Let $L$ and $J$ be  $m$-component links.  Then  $|\beta(L)-\beta(J)| \leq c(L,J)$.
Moreover, if $c(L,J) = |\beta(L) - \beta(J)|$, then the Witt  sum of Blanchfield forms $\Bl_L \oplus -\Bl_{J}$ is metabolic. }\\

Note that Corollary \ref{maincor} is an immediate consequence of Theorem \ref{mainthm} and Lemmas \ref{lem:sameorder} and  (\ref{lem:ordernorm}).

The first part of the theorem has been shown by Kawauchi \cite{Ka13}, however our argument will lead into the proof of the second part of the theorem, so we give a complete argument.

Let $L$ and $J$ be  $m$-component links. We write $c:=c(L,J)$.
We start out by introducing notation, especially for the combinatorics of an immersed concordance.

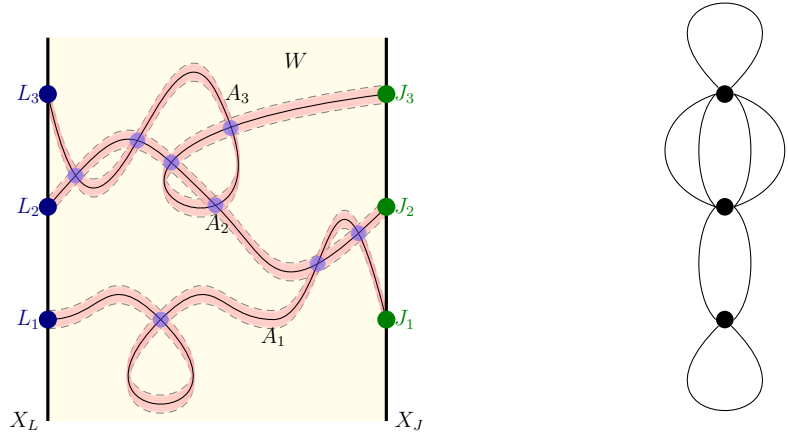
\begin{figure}
\begin{tikzpicture}
\begin{scope}[scale=1.5]
\fill[yellow!20,opacity=0.5] (-4,1.5) rectangle (-1,-1.9);
\draw (-1.8,1.3) node [scale=0.7] {$W$};
\foreach \y in {-0.08,0.08} {
\begin{scope}[shift={(0,\y)}, densely dashed]
\path[draw,opacity=0.5] (-4,1) .. controls (-3.5,-1.3) and (-3,2) .. (-2.5,1)  .. controls (-1.5,-1) and (-5,0.5) .. (-1,1);
\path[draw,opacity=0.5] (-4,0) .. controls (-3.3,0.8) .. (-2.5,0)  .. controls (-2,-0.5) and (-2,-1) .. (-1,0);
\path[draw,opacity=0.5] (-4,-1) .. controls (-3.5,-1) and (-3.5,-0.5) .. (-3,-1) .. controls (-2,-2) and (-4,-2) .. (-3,-1) .. controls (-2.5,-0.5) and (-2.5,-1).. (-2,-1) .. controls (-1.5,-1) and (-1.5,1) .. (-1,-1);
\end{scope}
}
\foreach \y in {-0.06,-0.04,-0.02,0,0.02,0.04,0.06} {
\begin{scope}[shift={(0,\y)}]
\path[draw,opacity=0.5,red!30,very thick] (-4,1) .. controls (-3.5,-1.3) and (-3,2) .. (-2.5,1)  .. controls (-1.5,-1) and (-5,0.5) .. (-1,1);
\path[draw,opacity=0.5,red!30,very thick] (-4,0) .. controls (-3.3,0.8) .. (-2.5,0)  .. controls (-2,-0.5) and (-2,-1) .. (-1,0);
\path[draw,opacity=0.5,red!30,very thick] (-4,-1) .. controls (-3.5,-1) and (-3.5,-0.5) .. (-3,-1) .. controls (-2,-2) and (-4,-2) .. (-3,-1) .. controls (-2.5,-0.5) and (-2.5,-1).. (-2,-1) .. controls (-1.5,-1) and (-1.5,1) .. (-1,-1);
\end{scope}
}
\path[draw,name path=a3] (-4,1) .. controls (-3.5,-1.3) and (-3,2) .. (-2.5,1) node[scale=0.7,above=0.22,right=0.05]{$A_3$} .. controls (-1.5,-1) and (-5,0.5) .. (-1,1);
\path[draw,name path=a2] (-4,0) .. controls (-3.3,0.8) .. (-2.5,0) node [right=0.01,below=0.06,scale=0.7]{$A_2$} .. controls (-2,-0.5) and (-2,-1) .. (-1,0);
\path[draw,name path=a1] (-4,-1) .. controls (-3.5,-1) and (-3.5,-0.5) .. (-3,-1) .. controls (-2,-2) and (-4,-2) .. (-3,-1) .. controls (-2.5,-0.5) and (-2.5,-1).. (-2,-1)
node[below=0.08, scale=0.7]{$A_1$} .. controls (-1.5,-1) and (-1.5,1) .. (-1,-1);
\fill[blue!70,opacity=0.5] (-3,-1) circle (0.07);
\fill[blue!70,opacity=0.5] (-2.38,0.7) circle (0.07);
\fill[name intersections={of=a1 and a2, by=x, name=int, total=\t}] [blue!70,opacity=0.5] \foreach \s in {1,...,\t}{ (int-\s) circle (0.07)};
\fill[name intersections={of=a3 and a2, by=x, name=int, total=\t}] [blue!70,opacity=0.5] \foreach \s in {1,...,\t}{ (int-\s) circle (0.07)};
\draw[very thick] (-4,1.5) -- (-4,-1.9) node[left=0.1,scale=0.7] {$X_L$};
\draw[very thick] (-1,1.5) -- (-1,-1.9) node[right=0.1,scale=0.7] {$X_J$};
\foreach \x/\xa in {-1/1,0/2, 1/3} {\fill[blue!50!black] (-4,\x) circle[radius=0.08] node[left=0.08,scale=0.7] {$L_{\xa}$}; \fill[green!50!black] (-1,\x) circle[radius=0.08] node[right=0.08,scale=0.7] {$J_{\xa}$};}
\node[circle, fill=black,  scale=0.6] (A3) at (2,1) {};
\node[circle, fill=black, scale=0.6] (A2) at (2,0) {};
\node[circle, fill=black, scale=0.6] (A1) at (2,-1) {};
\draw (A3.north west) .. controls +(-1,1) and +(1,1) ..  (A3.north east);
\draw (A3.west)  .. controls +(-0.6,-0.2) and +(-0.6,0.2) .. (A2.west);
\draw (A3.west)  .. controls +(-0.2,-0.2) and +(-0.2,0.2) .. (A2.west);
\draw (A3.east) .. controls +(0.6,-0.2) and +(0.6,0.2) ..(A2.east);
\draw (A3.east) .. controls +(0.2,-0.2) and +(0.2,0.2) .. (A2.east);
\draw (A2.west)  .. controls +(-0.2,-0.2) and +(-0.2,0.2) .. (A1.west);
\draw (A2.east) .. controls +(0.2,-0.2) and +(0.2,0.2) ..(A1.east);
\draw (A1.south west) .. controls +(-1,-1) and +(1,-1) ..  (A1.south east);
\end{scope}
\end{tikzpicture}
\caption{Left: an example of an immersed concordance. The picture is a sketch in dimension 2.
The annuli are represented by curves and the links are represented as points. The manifold~$W$ is the part of the picture outside the neighbourhood
of the annuli. The dashed curves represent $P$. Right: the corresponding graph $\Gamma$ (the labelling of the edges is not drawn).}\label{fig:notation1}
\end{figure}

\begin{itemize}
  \item Let $A_1,\dots,A_m \looparrowright S^3 \times I$ be $m$ properly immersed annuli giving an immersed concordance between $L$ and $J$ with $c$ double points.
  Define $A := \bigcup_{i=1}^m A_i$. to be their union.
   So $A \cap (S^3 \times \{0\}) = L_i$ and $A \cap (S^3 \times \{1\}) = -J$.
  (Here we identify $L\times \{0\}$ with $L$ and similarly we identify $J\times \{1\}$ with $J$.)
  \item Define $d_{ij} := |A_i \cap A_j|$ to be the number of double points between $A_i$ and $A_j$, for $i \neq j$.  Let $d_{ii}$ be the number of self-intersections of $A_i$. Note that $d_{ij} = d_{ji}$.  We have $\sum_{i \leq j} d_{ij} = c$.
  \item Let $\G$ be the graph defined by the combinatorics of the intersections of the $A_i$.  That is, take $m$ vertices corresponding to each of the annuli $A_i$, and add $d_{ij}$ edges between the $i$th and $j$th vertices.  There are $c$ edges in total.
  \item We introduce a labelling of edges of $\Gamma$. An edge corresponding to a positive double point is labelled with ``$+$'', while an edge corresponding
  to a negative double point is labelled with ``$-$''.
 \item Let $E := \beta_0(\G)$ be the number of connected components of $\G$ and let 
 $D:= 
 \beta_1(\G)$ be the first Betti number of~$\G$.
 We have $\chi(\G) = E-D = m-c$.
 \item Define $W := S^3 \times I \sm \nu A$ to be the exterior of the annuli $A$.
 \item The boundary of $W$ decomposes as $\partial W = X_L \cup_{\partial X_L} P \cup_{\partial X_J} X_J$, where this defines $P$.  That is, $P:= \cl( \partial W \sm (X_L \cup X_J))$, or equivalently $P = \partial \nu A \sm (\nu L \cup \nu J)$.
 \end{itemize}

Some of the notation is sketched in Figure~\ref{fig:notation1}.
We have three lemmata which lead to the proof of Theorem~\ref{mainthm}.  The first looks at the integral homology of $P$.

\begin{lemma}\label{lem:homology-of-P}
  The integral homology of 
$P$ is given by
$H_0(P;\Z) \cong \Z^E$, $H_1(P;\Z) \cong \Z^{2m+D}$ and $H_2(P;\Z) \cong \Z^{2m+D-E}$.
\end{lemma}

The next lemma computes the integral homology of $W$.

\begin{lemma}\label{lem:homology-of-W}
  The integral homology of $W$ is given by $$H_k(W;\Z) \cong \begin{cases}
    \Z & k=0 \\ \Z^{m} & k=1 \\ \Z^{m + c - 1} &k=2\\ 0 & \text{otherwise.}
  \end{cases}$$
Furthermore there exists an isomorphism $\phi_W \colon H_1(W;\Z) \toiso \Z^m$ such that the diagram below commutes, where the other maps are either induced by the inclusions or they are given by the canonical isomorphisms $\phi_L$ and $\phi_J$ induced by the orientations of the links:
\[\xymatrix @R+0.5cm @C+0.5cm {
 H_1(X_L;\Z) \ar[r]^-{\cong} \ar[dr]_-{\phi_L}^-{\cong} & H_1(W;\Z) \ar[d]_{\cong}^{\phi_W} & H_1(X_J;\Z) \ar[l]_-{\cong} \ar[ld]^-{\phi_J}_-{\cong} \\
 & \Z^m &
 }\]
 In particular $H_1(W;\Z)$ is generated by the meridians to $L$, or to $J$, and the maps $H_1(X_L;\Z) \to H_1(W;\Z)$ and $H_1(X_J;\Z) \to H_1(W;\Z)$ are isomorphisms.
\end{lemma}

For any subset $V \subseteq W$, a $\L$ coefficient system is defined with the representation $\pi_1(V) \to H_1(V;\Z) \to H_1(W;\Z) \xrightarrow{\phi_W} \Z^m$.

The third and final lemma needed for the proof of Theorem~\ref{mainthm} looks at the $\L$ coefficient homology of $P$, where the coefficient system is defined with the representation $\pi_1(P) \to H_1(P;\Z)  \to H_1(W;\Z) \xrightarrow{\phi_W} \Z^m.$

\begin{lemma}\label{lem:ord-Lambda-hom-P-negligible}
The homology $H_*(P;\Omega)$
is trivial. Moreover the order of the homology $\ord H_1(P;\L)$ is a negligible polynomial.
\end{lemma}

Next we give the proofs of the three lemmata above.

\begin{proof}[Proof of Lemma~\ref{lem:homology-of-P}]
The key observation is that since each $P$ is a boundary of a tubular neighbourhood of a
surface with double points, it is a (possibly disconnected) plumbed
3-manifold; see \cite[Example 4.6.2]{GS99}. We remark that if we add, to each vertex in the graph $\Gamma$, two edges ending in arrowhead vertices, then we obtain the plumbing diagram for~$P$
in the sense of \cite[Appendix]{Ne81}. Note that the framings are irrelevant because none of the plumbed components are closed.  Also note that in
our convention the plumbing corresponding to a disconnected graph is a disjoint union of the plumbed manifolds
corresponding to the connected components of the graph.  Elsewhere in the literature it has been the connected sum instead of a disjoint union.  Computation of $H_1(P;\Z)$ is a standard procedure.  We recall it for the reader's convenience and for future reference in the proof of Lemma~\ref{lem:ord-Lambda-hom-P-negligible}.

\begin{figure}
\begin{tikzpicture}
\draw [top color=red!50!white,bottom color=red!50!white, middle color=red!40!white, thick,draw=red!70!black](-2,1) .. controls (-1.7,0.5) and (-1.7,-0.5) .. (-2,-1) -- (-4,-1).. controls (-3.7,-0.5) and (-3.7,0.5) .. (-4,1) -- (-2,1);
\draw[top color=red!50!white,bottom color=red!50!white, middle color=red!40!white, thick,draw=red!70!black] (2,1) .. controls (2.3,0.5) and (2.3,-0.5) .. (2,-1) -- (4,-1).. controls (4.3,-0.5) and (4.3,0.5) .. (4,1) -- (2,1);
\draw[left color=blue!50!white,right color=blue!50!white, middle color=blue!40!white, thick,draw=blue!70!black] (1,2) .. controls (0.5,1.7) and (-0.5,1.7) .. (-1,2)
-- (-1,4) .. controls (-0.5,3.7) and (0.5,3.7) .. (1,4) -- (1,2);
\draw[left color=blue!50!white,right color=blue!50!white, middle color=blue!40!white, thick,draw=blue!70!black] (1,-2) .. controls (0.5,-2.3) and (-0.5,-2.3) .. (-1,-2)
-- (-1,-4) .. controls (-0.5,-4.3) and (0.5,-4.3) .. (1,-4) -- (1,2);
\draw[very thick,green!70!black,inner color=green!30!white, outer color=green!50!white] (-2,-1) arc [start angle=90, delta angle=-90, radius=1] .. controls (-0.5,-2.3) and (0.5,-2.3)  .. (1,-2)
arc [start angle=180, delta angle=-90, radius=1] .. controls (2.3,-0.5) and (2.3,0.5) .. (2,1) arc[start angle=270, delta angle=-90, radius =1]
.. controls (0.5,1.7) and (-0.5,1.7) .. (-1,2) arc [start angle=0, delta angle=-90, radius=1] .. controls (-1.7,0.5) and (-1.7,-0.5) .. (-2,-1);
\shade[left color=blue!20!white, right color=blue!20!white, middle color=blue!15!white] (-1,4) .. controls (-0.5,3.7) and (0.5,3.7) .. (1,4) .. controls (0.5,4.2)
and (-0.5,4.2) .. (-1,4);
\shade[top color=red!20!white, bottom color=red!20!white, middle color=red!15!white] (-4,1) .. controls (-3.7,0.5) and (-3.7,-0.5) .. (-4,-1) .. controls (-4.2,-0.5)
and (-4.2,0.5) .. (-4,1);
\draw[very thick, densely dashed] (-1,-2) .. controls (-0.5,-1.8) and (0.5,-1.8) .. (1,-2);
\draw[very thick, densely dashed] (-1,2) .. controls (-0.5,2.2) and (0.5,2.2) .. (1,2);
\draw[very thick, densely dashed] (-1,4) .. controls (-0.5,4.2) and (0.5,4.2) .. (1,4);
\draw[very thick, densely dashed] (-1,-4) .. controls (-0.5,-3.8) and (0.5,-3.8) .. (1,-4);
\draw[very thick] (-1,2) .. controls (-0.5,1.7) and (0.5,1.7) .. (1,2);
\draw[very thick] (-1,-2) .. controls (-0.5,-2.3) and (0.5,-2.3) .. (1,-2);
\draw[very thick] (-1,-4) .. controls (-0.5,-4.3) and (0.5,-4.3) .. (1,-4);
\draw[very thick] (-1,4) .. controls (-0.5,3.7) and (0.5,3.7) .. (1,4);
\foreach \x/\xt in {-4/-3.7,-2/-1.7,2/2.3,4/4.3} \draw[very thick] (\x,-1) .. controls (\xt,-0.5) and (\xt,0.5) .. (\x,1);
\foreach \x/\xt in {-4/-4.2,-2/-2.2,2/1.8,4/3.8} \draw[very thick, densely dashed] (\x,-1) .. controls (\xt,-0.5) and (\xt,0.5) .. (\x,1);
\begin{scope}[ultra thick]
\draw (-4.5,0) -- (-3.9,0);
\draw (-3.6,0) -- (-1.9,0);
\draw (-1.6,0) -- (2.1,0);
\draw (2.4,0) -- (4.1,0);
\draw (4.4,0) -- (4.9,0);
\draw (0,-4.9) -- (0,-4.4);
\draw (0,-4.1) -- (0,-2.4);
\draw (0,-2.1) -- (0,1.6);
\draw (0,1.9) -- (0,3.6);
\draw (0,3.9) -- (0,4.5);
\draw (4.7,0.2) node[scale=0.8] {$A_1$};
\draw (0.28,-4.7) node [scale=0.8] {$A_2$};
\draw (1.5,1.5) node [scale=0.8] {$T$};
\draw (1.25,3) node [scale=0.8] {$Y_2$};
\draw (1.25,-3) node [scale=0.8] {$Y_2$};
\draw (3,1.2) node [scale=0.8] {$Y_1$};
\draw (-3,1.2) node [scale=0.8] {$Y_1$};
\end{scope}
\fill (0,0) circle [radius=0.1];
\end{tikzpicture}
\caption{A schematic of the manifold $P$ near a double point. Two annuli, here $A_1$ and $A_2$, intersect at a single point in the middle. The part of $P$ near this point
(denoted by $T$ in the picture) is the complement of the Hopf link, that is, of the link of the singularity of type `ordinary double point'. }\label{fig:nearadoublepoint}
\end{figure}
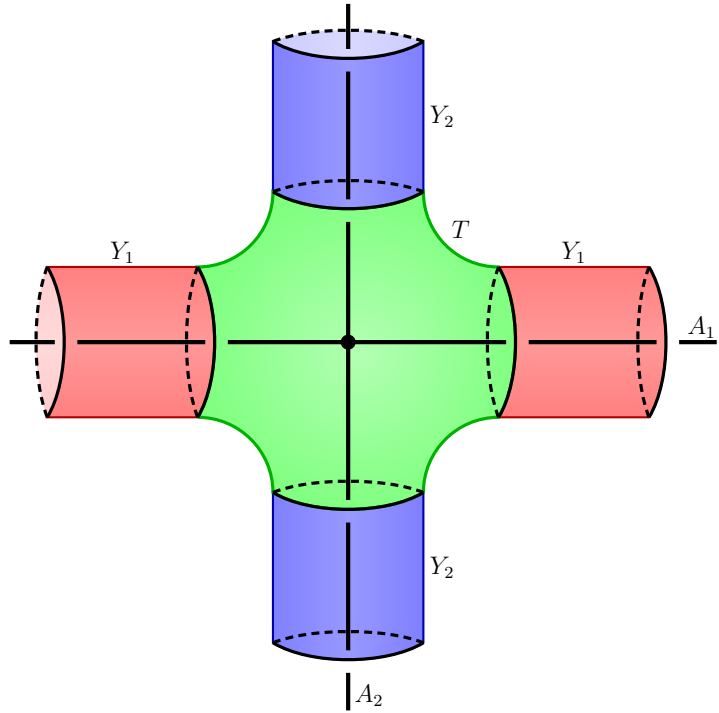
Let $\S$ be a disjoint union of $m$ annuli. If there are no double points, then $P=\S\times S^1$. Otherwise we construct $P$ as follows; compare
\cite[Section 1]{Ne81} and see Figure~\ref{fig:nearadoublepoint}.
For each vertex of the
graph $\Gamma$ we take an annulus with as many discs removed, as is the valency (number of incident edges) of the vertex. Let $Y_1,\ldots,Y_m$ be these punctured
annuli. The number of removed discs in $Y_j$ is $\Delta_j:=d_{jj}+\sum_{i=1}^m d_{ij}$. Let $\p_{j1},\ldots,\p_{j\Delta_j}$ be the boundary components of $Y_j$
corresponding to these discs. The total number of these boundary components
is equal to $2c$.  (Note that each $Y_j$ has two additional boundary components, namely the boundary of the annulus.)

In the reconstruction of $P$ it is convenient to temporarily orient the edges of $\Gamma$, however the output is independent of this orientation.
For each edge of $\Gamma$ take a torus  $S^1\times S^1$. Let $T_1,\ldots,T_c$ be these tori.
Then the manifold $P$ is a union of the products $Y_j\times S^1$ and thickened tori $T_i\times[-1,1]$. The glueing data is encoded in the graph $\Gamma$.
Namely, if the edge corresponding
to $T_i$ starts at the vertex corresponding to $Y_j$, we identify $T_i\times\{-1\}$ with $\p_{jk}\times S^1\subset\partial Y_j\times S^1$, where $k$ is determined
by the combinatorics of the graph. If the edge corresponding to $T_i$ ends at the vertex corresponding to $Y_j$, we identify $T_i\times\{1\}$ with $S^1\times \p_{jk}$.
We remark that this time the coordinates are swapped. Moreover, if the edge is marked with a ``$-$'', the last identification reverses the orientation both of $S^1$ and of $\p_{jk}$;
see \cite[Section 1]{Ne81}.

Denote the canonical maps $T_i\to T_i\times \{\pm 1\}$ by $\iota_{\pm}$.
The Mayer--Vietoris sequence thus gives rise to the following long exact sequence
\begin{equation}\label{eq:mvforplumbing}
\xymatrix@C0.8cm@R-1.0cm{ \dots\ar[r]& \bigoplus\limits_{i=1}^c H_1(T_i;\Z)\ar[r]^-{\iota_--\iota_+} &H_1\big(Y\times S^1;\Z\big)\ar[r] & H_1 (P;\Z)\ar[r]&\\
\ar[r]& \bigoplus\limits_{i=1}^c H_0(T_i;\Z)\ar[r]^-{\iota_--\iota_+} & H_0\big(Y \times S^1;\Z\big)\ar[r]& H_0(P;\Z)\ar[r]&0}
\end{equation}
where $Y := \bigcup_{i=1}^m Y_i$.
Now we have the following claim.

\begin{claim}
The homomorphism
\[  \bigoplus\limits_{i=1}^c H_1(T_i;\Z)\xrightarrow{\iota_--\iota_+} H_1\big(Y\times S^1;\Z\big)\]
splits.
\end{claim}

A straightforward argument shows that
the curves $\p_{11},\ldots,\p_{m\Delta_m}$ (their number is $2c$) freely generate a summand
 of $H_1(Y;\Z)$. In particular there is a splitting
\[  s\co H_1(Y;\Z)\to \bigoplus_{\substack{i=1,\ldots,m\\j=1,\ldots,\Delta_i}} \Z \p_{ij}.\]
It follows easily from the glueings that the map
\[ \bigoplus\limits_{i=1}^c H_1(T_i;\Z)\xrightarrow{\iota_--\iota_+} H_1\big(Y\times S^1;\Z\big)
\to H_1(Y;\Z)\xrightarrow{s}\bigoplus_{i,j} \Z \p_{ij}\]
is an isomorphism. This concludes the proof of the claim.

The lemma is now an immediate consequence of the exact sequence
(\ref{eq:mvforplumbing}), the definitions and the fact that the above homomorphism splits. Indeed, the $H_0$ terms in the exact sequence are precisely the terms and maps which compute $H_0(\Gamma)$. Since $H_0(\Gamma;\Z)\cong \Z^E$,
the Mayer--Vietoris exact sequence \eqref{eq:mvforplumbing} reduces to
\[  \Z^{2c}\to \Z^{2c+2m}\to H_1(P;\Z)\to \Z^c\to \Z^m\to \Z^E\to 0\]
where the left hand homomorphism splits. We thus see that $H_1(P;\Z)\cong \Z^{2m}\oplus \Z^{c-m+E}$.
By $E-D=m-c$ it follows that $H_1(P)\cong \Z^{2m+D}$. The statement that $H_2(P;\Z)\cong \Z^{2m+D-E}$ also follows from the Mayer--Vietoris sequence. Alternatively it follows from an Euler characteristic argument and the observation that $H_2(P;\Z)\cong H^1(P,\partial P;\Z)=\hom(H_1(P,\partial P;\Z),\Z)$ is torsion-free.
\end{proof}

\begin{proof}[Proof of Lemma~\ref{lem:homology-of-W}]
To compute the integral homology of $W$ we use the Mayer--Vietoris sequence associated to the decomposition $S^3 \times I \cong W \cup_P \nu A$.  For $H_1(W;\Z)$, we have an exact sequence:
\[H_2(S^3 \times I;\Z) \to H_1(P;\Z) \to H_1(W;\Z) \oplus H_1(\nu A;\Z) \to H_1(S^3 \times I;\Z).\]
As $\nu A$ strongly retracts onto $A$,
there is an homotopy equivalence
$$\nu A \simeq ((\Gamma \vee S^1)\vee S^1) \vee \dots \vee S^1,$$
where there are $m$ copies of $S^1$, one attached to each vertex of $\G$.  That is, we change the basepoint for each wedge sum.  Therefore $H_1(\nu A;\Z) \cong \Z^{m+D}$ and the exact sequence above becomes:
\[0 \to \Z^{2m+D} \to H_1(W;\Z) \oplus \Z^{m+D} \to 0.\]
It follows that $H_1(W;\Z) \cong \Z^m \cong \ker(H_1(P;\Z) \to H_1(\nu A;\Z))$.  This kernel is freely generated by the meridians to $L$ (or to $J$).  Therefore the inclusion induced maps $H_1(X_L;\Z) \to H_1(W;\Z)$ and $H_1(X_L;\Z) \to H_1(W;\Z)$ are isomorphisms.  It is now straightforward to see that the
maps $H_1(W;\Z)\xleftarrow{\cong}H_1(X_L;\Z)\xrightarrow{\phi_L}\Z^m$ and
$H_1(W;\Z)\xleftarrow{\cong}H_1(X_J;\Z)\xrightarrow{\phi_J}\Z^m$ agree.
We denote this isomorphism $H_1(W;\Z)\to \Z^m$ by $\phi_W$.
By Poincar\'{e}-Lefschetz duality, $H_3(W;\Z) \cong H^1(W;\partial W;\Z)$, which fits into the short exact sequence
\[0 \to \Ext^1_{\Z}(H_0(W,\partial W;\Z),\Z) \to H^1(W,\partial W;\Z)  \to \Ext^0_{\Z}(H_1(W,\partial W;\Z),\Z)  \to 0\]
by the universal coefficient theorem.  Since $\partial W$ is connected we have an isomorphism $H_0(\partial W;\Z) \toiso H_0(W;\Z)$.  Therefore $H_0(W,\partial W;\Z) =0$
and the map $H_1(W;\Z) \to H_1(W,\partial W;\Z)$ is surjective.  However we just saw that the meridians of $L$ generate  $H_1(W;\Z)$.  Since the meridians of~$L$ lie in $\partial W$, the image of the map $H_1(W;\Z) \to H_1(W,\partial W;\Z)$ is zero.  Thus $H_1(W,\partial W;\Z)$ vanishes, from which we see that $H_3(W;\Z) =0$.

Next note that $H_2(W;\Z)$ is torsion free.  To see this, observe that the torsion subgroup of $H_2(W;\Z)$ is a subgroup of $H^3(W;\Z)$ by the universal coefficient theorem, but $H^3(W;\Z) \cong H_1(W,\partial W;\Z) =0$ by Poincar\'{e}-Lefschetz duality.  Therefore to find $H_2(W;\Z)$ it suffices to know its rank.

Now we may compute with the Euler characteristic.  First $\chi(S^3 \times I) = \chi(S^3) = 0$.  Also by Lemma \ref{lem:homology-of-P} we have $\chi(P) = 0$, so $0 = \chi(S^3 \times I) = \chi(W) + \chi(\nu A) - \chi(P) = \chi(W) + \chi(\nu A)$.  As above, $\nu A$ is homotopy equivalent to a graph $\G'$ with $\beta_0(\G')=E$ and $\beta_1(\G') = m+D$.  Therefore $\chi(\nu A) = E-D-m$, from which we see that $\chi(W) = m+D-E$.   Then note that $\chi(\G) = E-D = m-c$, computing in the first instance using the Betti numbers of $\G$ and in the second instance using the fact that there are $m$ 0-cells and $c$ 1-cells in a cell decomposition of $\G$.  Thus $$\chi(W) = m+D-E = m+c-m = c.$$
From this we may compute the rank of $H_2(W;\Z)$. Since $\beta_0(W) =1$ and $\beta_1(W) = m$, we have $c=\chi(W) = 1-m+\rk H_2(W;\Z)$, so that $H_2(W;\Z) \cong \Z^{m+c-1}$ as claimed.
\end{proof}

\begin{proof}[Proof of Lemma~\ref{lem:ord-Lambda-hom-P-negligible}]
In the proof of Lemma~\ref{lem:homology-of-P} we constructed $P$ as a union of $Y_1\times S^1,\ldots,Y_m\times S^1$, where $Y_1,\ldots,Y_m$ are punctured
annuli, and thickened tori $T_1\times[-1,1],\ldots,T_c\times[-1,1]$. In the present proof we use the same description of $P$.

First we show that $H_*(P;\O)=0$. To this end, consider the short exact sequence of chain complexes
\begin{multline*}
0 \to \bigoplus_{\substack{i=1,\ldots,c\\u=\pm 1}} C_*(T_i\times\{u\};\O) \to \\[-5mm]
\to \bigoplus_{i=1}^m C_*(Y_i \times S^1;\Omega) \oplus \bigoplus_{i=1,\ldots,c} C_*(T_i;\O)
 \to C_*(P;\Omega) \to  0. \end{multline*}
The $\O$ coefficient system for any subset $V \subseteq P$ is defined via the map $\pi_1(V) \to \pi_1(P) \to \L \to \O$.  All of these chain complexes are acyclic.  To see this, note that by the associated Mayer--Vietoris sequence, this will follow once we see that the $\O$ coefficient homology is trivial for all complexes apart from $C_*(P;\O)$.  The homology $H_*(S^1;\O)=0$ whenever the generator of $\pi_1(S^1)$ maps nontrivially into $\O$.  Then $H_*(X \times S^1;\O) =0$ for any $X$ by
Lemma \ref{lemma:torsion-X-times-S1}.
This accounts for all the remaining terms, and thus completes the proof of the first part of the lemma, that $H_*(P;\O) =0$.

Given that $H_*(P;\O)=0$, we can compute the Reidemeister torsion of $P$;
from this we will be able to deduce the order of the first homology with $\L$ coefficients.
The representation $\pi_1(P) \to H_1(P;\Z) \to H_1(W;\Z) \to \Z^m = \ll t_1,\dots,t_m \rr$ sends a meridian of $A_i$ to $t_i$. Therefore we may apply
Lemma~\ref{lemma:torsion-X-times-S1} to see that
\[\tau(Y_i \times S^1) = (t_i-1)^{-\chi(Y_i)} = (t_i-1)^{\Delta_i}\]
since $Y_i$ is an annulus with $\Delta_i := d_{ii}+\sum_{j=1}^m d_{ij}$ punctures.  Both $S^1 \times S^1$ and $S^1 \times S^1 \times I$ have Reidemeister torsion $1$, by a further application of Lemma~\ref{lemma:torsion-X-times-S1} with $X=S^1$ and $X=S^1 \times I$ respectively.

As $P$ is presented as a union of thickened tori $T_1\times[-1,1],\ldots,T_c\times[-1,1]$ and $Y_1\times S^1,\ldots,Y_m\times S^1$ along tori,
    the glueing formula of Theorem~\ref{thm:mult_of_torsions_in_ses} yields the formula
\be \label{equ:torsionp} \tau(P) = \prod_{i=1}^m (t_i-1)^{\Delta_i},\ee
in particular  $\tau(P)$ is negligible. By \cite[Theorem~4.7]{Tu01} we have
\[\tau(P) = \prod_{i=0}^2 (\ord H_i(P;\L))^{(-1)^i}.\]
Thus it suffices to show that $\ord H_0(P;\L)$ and $\ord H_2(P;\L)$ are negligible.

It follows immediately from Lemma \ref{lem:h0} and from the definitions that the order of  $H_0(P;\L)$  is $1$ if $m \geq 2$, and $t-1$ if $m=1$. In both cases $\ord H_0(P;\L)$ is negligible.
Now we turn to $H_2(P;\L)$. Note that $P$ is homotopy equivalent to a 2-complex $P'$. Therefore $H_2(P;\L)=H_2(P';\L)$ is a submodule of the free $\L$-module $C_2(P';\L)$. In particular $H_2(P;\L)$ is torsion-free.
By the first part of the lemma
 we know that $H_2(P;\L)\otimes_{\L}\Omega=H_2(P;\Omega)=0$. It follows that $H_2(P;\L)=0$, in particular $\ord H_2(P;\L)=1$.
\end{proof}

\begin{remark}
  Since $(t_i-1)^2$ is a norm, and each self-intersection contributes $(t_i-1)^2$, we see the linking number differences between the two links $L$ and $J$ determine the negligible terms up to norms, just as in~\cite{Ka13}.
\end{remark}

\smallskip

Now we have assembled the necessary ingredients, we throw them into the pre-heated sizzling pan of long exact sequences that is the proof of Theorem~\ref{mainthm}.

\begin{proof}[Proof of Theorem~\ref{mainthm}]
Let $L$ and $J$ be  $m$-component links. We write $c=c(L,J)$.
We begin by studying the ranks $\b(L)$ and $\b(J)$. Without loss of generality we can assume that $\b(J)\geq \b(L)$.  We then have the following claim which in particular proves the first statement of the theorem.

\begin{claim}
We have $\b(J)\in \{\b(L),\dots,\b(L)+c\}$.
\end{claim}

Consider the long exact sequences of the pairs $(W,X_L)$ and $(W,X_{J})$ with $\Omega$ coefficients:
\begin{eqnarray*}
 \dots  \longrightarrow & H_3(X_L;\O) \longrightarrow & H_3(W;\O) \longrightarrow H_3(W,X_L;\O) \longrightarrow \\
 \longrightarrow &H_2(X_L;\O)  \longrightarrow & H_2(W;\O) \longrightarrow H_2(W,X_L;\O) \longrightarrow \\
 \longrightarrow & H_1(X_L;\O)  \longrightarrow & H_1(W;\O) \longrightarrow H_1(W,X_L;\O) \longrightarrow \dots \\
\end{eqnarray*}
and
\begin{eqnarray*}
\dots \longrightarrow & H_3(X_J;\O) \longrightarrow & H_3(W;\O) \longrightarrow H_3(W,X_J;\O) \longrightarrow \\
 \longrightarrow &H_2(X_J;\O)  \longrightarrow & H_2(W;\O) \longrightarrow H_2(W,X_J;\O) \longrightarrow \\
 \longrightarrow & H_1(X_J;\O)  \longrightarrow & H_1(W;\O) \longrightarrow H_1(W,X_J;\O) \longrightarrow \dots \\
\end{eqnarray*}

We investigate the dimensions of the terms in these sequences.  Let
\[\beta:= \beta(L) = \dim H_1(X_L;\O).\]
By Lemma \ref{lem:h0} we have $H_0(X_L;\O)= 0$. Since $X_L$ is homotopy equivalent to a 2-complex, we have $H_3(X_L;\O) =0$.  The Euler characteristic of $X_L$ is zero since $X_L$ is a 3-manifold with a toroidal boundary.  Therefore $\dim H_2(X_L;\O) =\beta$.

Next by Lemma~\ref{lem:homology-of-W} we have $H_i(W,X_L;\Z) =0$ for $i=0,1$.
Therefore $H_1(W,X_L;\O) =0$ by \cite[Proposition~2.10]{COT03}.

\begin{claim}
We have  $H_3(W,X_J;\Omega)=0$.
\end{claim}

By Theorem \ref{thm:pd} the module  $H_3(W,X_J;\Omega)$
is  isomorphic to $H^1(W,\ol{\partial W\sm X_J};\O) \cong H_1(W,\ol{\partial W\sm X_J};\O)$. Appealing again to
\cite[Proposition~2.10]{COT03} we see that
  $H_3(W,X_J;\Omega) =0$  if $H_1(W,\ol{\partial W\sm X_J};\Z) =0$. To see that  $H_1(W,\ol{\partial W\sm X_J};\Z) =0$, we consider the long exact sequence of the triple $(W, \ol{\partial W\sm X_J}, X_L)$:
\[H_1(\ol{\partial W\sm X_J}, X_L;\Z) \to H_1(W,X_L;\Z) \to H_1(W, \ol{\partial W\sm X_J};\Z) \to H_0(\ol{\partial W\sm X_J}, X_L;\Z). \]
As we saw above, by Lemma~\ref{lem:homology-of-W}, $H_1(W,X_L;\Z) =0$.  Also $\ol{\partial W\sm X_J}$ is connected, so $H_0(\ol{\partial W\sm X_J}, X_L;\Z) =0$.  Thus $H_1(W,\ol{\partial W\sm X_J};\Z) =0$, as desired. This concludes the proof of the claim.

Using the exact sequence of the pair $(W,X_J)$ with $\Omega$ coefficients
provided  above and the facts that $H_3(W,X_J;\Omega) =0 \cong H_3(X_J;\O)$, we find that $H_3(W;\O) =0$.
We may also reverse the r\^{o}les of $L$ and $J$, so that also $H_3(W,X_L;\Omega) =0$.

Suppose that $\dim H_1(X_J;\O) = \beta + \ell$, where $\ell >0$: recall that without loss of generality we supposed that $\beta(J) \ge \beta(L)$.
It follows from $H_0(X_J;\O)=0$ and the usual Euler characteristic argument that $\dim H_2(X_J;\O) = \beta + \ell$.
  Next, since $H_1(W,X_L;\O) =0$ the map $$H_1(X_L;\O) \to H_1(W;\O)$$ is a surjection, so $\dim H_1(W;\O) \leq \beta$.  We also see that $H_3(W,X_J;\O)  =0$ implies $$H_2(X_J;\O) \to H_2(W;\O)$$ is an injection, so $\dim H_2(W;\O) \geq \beta + \ell$.

The only potentially nontrivial homology groups of $W$ with $\Omega$ coefficients are $H_1(W;\O)$ and $H_2(W;\O)$, since $H_0(W;\O) =0$, again by Lemma \ref{lem:h0}.  The Euler characteristic of $W$ is $\chi(W) = c$ by Lemma~\ref{lem:homology-of-W}, from which it follows that $\dim H_2(W;\O) = \dim H_1(W;\O) + c$.  Combining this with the fact that $\dim H_1(W;\O) \leq \beta$ yields $\dim H_2(W;\O) \leq \beta +c$.   Together the inequalities $$\beta +\ell \leq \dim H_2(W;\O) \text{ and } \dim H_2(W;\O) \leq \beta +c$$ yield $\ell \leq c$, which says that $\beta(J) - \beta(L) \leq c$.  We assumed without loss of generality that $\beta(J) \geq \beta(L)$, so this  completes the proof of the claim above that $\b(J)\in \{\b(L),\dots,\b(L)+c\}$, and therefore also the proof of the first part of the theorem.

\smallskip

Now we turn to the proof of the second statement.
First we note that  the Witt sum $\Bl_{L} \oplus -\Bl_J$ is isomorphic to $\Bl_{\partial W}$.  Indeed, by  Lemma~\ref{lem:ord-Lambda-hom-P-negligible} the  homology of $P$ with $\L_S$ coefficients is trivial. Using this observation, the argument on \cite[p.~39]{Hi12} carries over to give the desired statement on Blanchfield forms.
We leave the details to the reader.

Thus in light of Proposition \ref{prop:hillman-metabolic}, in order to see that the Witt sum of Blanchfield forms of the links $L$ and $-J$ is metabolic, it suffices to prove the following claim.

\begin{claim}
If $\beta(J) = c(L,J) + \beta(L)$, then the sequence
\[TH_2(W,\partial W;\L_S) \to TH_1(\partial W;\L_S) \to TH_1(W;\L_S)\]
is exact.
\end{claim}

In the notation of our proof the assumption that $\beta(J) = c(L,J) + \beta(L)$ implies that $\beta + \ell = c + \beta$, so that $c=\ell$. Therefore $\dim H_2(W;\Omega) = \beta+c$.  The Euler characteristic implies that $\dim H_1(W;\O) = \beta$.  Now we consider (\ref{eqn:les-W-partial-W-Omega-coeffs}), i.e.\ the long exact sequence of the pair $(W,\partial W)$ with $\Omega$ coefficients.  Underneath each entry we write its dimension, for the convenience of the reader, which we will then proceed to justify.

\begin{equation}\label{eqn:les-W-partial-W-Omega-coeffs}
\begin{array}{ccccccccc}
0\longrightarrow&H_3(W,\partial W;\Omega) & \longrightarrow & H_2(\partial W;\Omega) &\longrightarrow& H_2(W;\Omega) & \longrightarrow \\
&\beta &&2\beta +c&&\beta +c&\\
&H_2(W,\partial W;\Omega) & \longrightarrow & H_1(\partial W;\Omega) &\longrightarrow& H_1(W;\Omega) & \longrightarrow &0\\
&\beta +c&&2\beta +c&&\beta &
\end{array}
\end{equation}

By Theorem \ref{thm:pd} and by the above calculations we have $\dim H_3(W,\partial W;\O) = \dim H_1(W;\O) = \beta$ and $\dim H_2(W,\partial W;\O) = \dim H_2(W;\O) = \beta+c$.

Finally we also have  $\dim H_1(\partial W;\O) =2\b+c$. Indeed, by Lemma \ref{lemma:torsion-X-times-S1} and Lemma \ref{equ:torsionp} we have $H_*(\partial X_L;\Omega)=H_*(\partial X_J;\Omega)=H_*(P;\Omega)=0$. The  Mayer--Vietoris sequence for $\partial W = X_L \cup P \cup X_J$  with $\Omega$ coefficients then implies the desired equality
\[  \dim H_1(\partial W;\O) = \dim H_1(X_L;\O)  + \dim H_1(X_J;\O)    = 2\beta + c.\]

A quick look at the dimensions in the  long exact sequence (\ref{eqn:les-W-partial-W-Omega-coeffs})
  shows that the long exact sequence splits into two short exact sequences.
Now consider the following commutative diagram:
\[ \xymatrix{
&0\ar[d]&0\ar[d]&0\ar[d]\\
&TH_2(W,\partial W;\L_S)\ar[d]\ar[r] &TH_1(\partial W;\L_S)\ar[d]\ar[r] &TH_1(W;\L_S)\ar[d] \\
&H_2(W,\partial W;\L_S)\ar[d]\ar[r] &H_1(\partial W;\L_S)\ar[d]\ar[r] &H_1(W;\L_S)\ar[d] \\
0\ar[r]&H_2(W,\partial W;\Omega) \ar[r]&H_1(\partial W;\Omega)\ar[r] &H_1(W;\Omega).}\]
Note that the vertical sequences are exact. Also note that the middle horizontal sequence is exact. Furthermore, we have just shown that the  bottom  horizontal sequence is also exact. It follows from elementary diagram chasing (this is known as the sharp $3 \times 3$ lemma~\cite[Lemma~2]{FHH89}) that the top horizontal sequence is also exact.
\end{proof}

\section{The Gordian distance between links}
\label{section:gordian}

\subsection{Proof  of Theorem \ref{lem:torsionpolynomials}}
For the reader's convenience we recall the statement of Theorem \ref{lem:torsionpolynomials}.
\\

\noindent \textbf{Theorem \ref{lem:torsionpolynomials}.}\emph{
Let $L$ and $J$ be two $m$-component links.
Then
\[ |\b(L)-\b(J)|\leq g(L,J).\]
Furthermore, if  $\b(J)=\b(L)+g(L,J)$, then
\[  \deltator_L=\deltator_{J}\cdot f\, \ol{f}\cdot n\]
for some $f\in \L$ and some negligible  $n\in\L$. In particular
$\deltator_{L}$ divides $\deltator_{J}$.}\\

\begin{proof}
We write $L'=J$.
In light of Theorem \ref{mainthm} and the inequality $c(L,L')\leq g(L,L')$ it suffices to prove the second statement.
Let $L$ and $L'$ be two $m$-component links with $\b(L')=\b(L)+g(L,L')$.
We have to show that
\[ \deltator_L=\deltator_{L'}\cdot f\, \ol{f}\cdot n\]
for some $f\in \L$ and some negligible $n\in \L$.

We first consider the case that $g(L,L')=1$. We start out with the following claim.

\begin{claim}
There exists a non-zero $p\in \L$ such that
\[  \deltator_{L}=\deltator_{L'}\cdot p.\]
\end{claim}

We write $H=H_1(S^3\sm \nu L;\L)$,  $H'=H_1(S^3\sm \nu L';\L)$ and $\b=\b(L)$. By assumption we have $\rk_{\L}(H)=\b$ and $\rk_{\L}(H')=\b+1$.
In \cite[Proposition~4.1]{CFP13} we showed that there exists a diagram
\[\xymatrix @R-0.1cm { & \Lambda  \ar[d]_{f} &  & \\
\Lambda \ar[r]^{g} & M \ar[r]^{p'} \ar[d]_{p} & H' \ar[r] & 0   \\
 & H \ar[d]  & & \\
  & 0 & &}\]
where $M$ is some $\L$-module and where the horizontal and vertical sequences are exact. 
It follows from the horizontal exact sequence that $\rk_{\L}(M)\geq \b+1$. On the other hand from
considering the vertical exact sequence we see that $\rk_{\L}(M)\leq \b+1$.
Thus we deduce that $\rk_{\L}(M)=\b+1$. It then follows again from the vertical sequence that  $f$ is injective, which in turn  implies that
$ TM\to TH$ is a monomorphism. By Lemma \ref{lem:ordermultiplicative} we have that
\begin{equation} \label{equ:div1} \ord(TM) \,|\, \ord(TH).\end{equation}
 Consider the following commutative diagram
\[ \xymatrix{ 0\ar[r] & TM\ar[d]^{p'|}\ar[r]& M\ar[d]^{p'}\ar[r] & M\otimes_{\L}\Omega \ar[d]^{p' \otimes \id} \\
0\ar[r] & TH'\ar[r]& H'\ar[r] & H'\otimes_{\L}\Omega .}\]
The middle vertical map is an epimorphism and the right hand map is a monomorphism since $p'$ is a surjective homomorphism between two $\Omega$-vector spaces of the same dimension. Some mild diagram chasing shows that $p\colon TM\to TH'$ is an epimorphism.  Lemma~\ref{lem:ordermultiplicative}  then implies that
\begin{equation} \label{equ:div2}  \ord(TH') \,|\, \ord(TM).\end{equation}
The combination of (\ref{equ:div1}) and (\ref{equ:div2}) implies that
\[  \ord(TH') \,|\, \ord(TH).\]
 But this is exactly the desired statement. This concludes the proof of the claim.

\smallskip
We just showed that
$\deltator_{L}=\deltator_{L'}\cdot p$ for some non-zero $p\in \L$.
Moreover by Corollary \ref{maincor} we know  that
\[  \deltator_{L}\cdot g\,\ol{g}= \deltator_{L'}\cdot g'\,\ol{g'}\cdot  n\]
for some $g,g'\in \L$ and some negligible~$n$. If we combine these two statements we see that $g\,\ol{g}$ divides $g'\,\ol{g'}\cdot n$. Since $\L$ is a UFD we have that
$g'\,\ol{g'}\cdot n=g\,\ol{g}\cdot f\,\ol{f}\cdot m$ for some $f\in \L$ and some negligible~$m$. Simplifying, we obtain
$ \deltator_{L}= \deltator_{L'}\cdot f\,\ol{f}\cdot  m$.
This concludes the proof of the theorem in the case $g(L,L')=1$.

\smallskip
Now suppose that $g(L,L')=g>1$. Then there exists a sequence $L=L_0,L_1,\dots,L_g=L'$ of links such that each $L_i$ is obtained from the previous link by a single crossing change. By Theorem \ref{mainthm}  we have $|\b(L_{i+1})-\b(L_i)|\leq 1$ for each $i$.
It follows from the assumption  $\b(L')=\b(L)+g(L,L')$ that for each $i$ we have in fact  $\b(L_{i+1})=\b(L_i)+1$.
The desired statement follows easily  from applying the above result to the $g$ pairs of links.
\end{proof}

\subsection{Applications of Theorem \ref{lem:torsionpolynomials}}
In this section we will discuss applications of  Theorem \ref{lem:torsionpolynomials} to various special cases of determining the Gordian distance between links.  We start out with the following well-known lemma.

\begin{lemma} \label{lem:rankmminus1}
For an $m$-component link $L$ we have $\b(L)\leq m-1$.
\end{lemma}

\begin{proof}
The statement of the lemma is well-known to the experts, we will therefore just provide a sketch of an argument.
Let $L$ be an $m$-component link.
  Consider the inclusion of a wedge of $m$ circles $Y:=\bigvee_m S^1 \to X_L$ which sends each circle to a meridian of a different component of $L$.  The induced map on zeroth and first homology is an isomorphism.  In particular $H_i(X_L,\bigvee_m S^1;\Z) =0$ for $i=0,1$. It follows from \cite[Proposition~2.10]{COT03} that $H_1(X_L,Y;\O)=0$, which in turn implies that $H_1(Y;\Omega)\to H_1(X_L;\Omega)$ is surjective. Thus it suffices to show that $H_1(Y;\Omega) \cong \Omega^{m-1}$.  Note that by Lemma \ref{lem:h0} we have $H_0(Y;\Omega)=0$, therefore an Euler characteristic argument shows that indeed $H_1(Y;\Omega) \cong \Omega^{m-1}$.
\end{proof}

The following corollary to Theorem \ref{lem:torsionpolynomials} says in particular that the gap between the rank $\b(L)$ of the Alexander module  and the maximal possible rank $m-1$ gives a lower bound on the unknotting number.
Note that this particular corollary is in fact a special case of \cite[Theorem~1.1]{Ka13}.

\begin{corollary}\label{cor:unlink}
Let $L$ be an $m$-component link. Then the following hold:
\bn
\item We have $m-1-\b(L)\leq u(L)$.
 In particular if $\Delta_L\ne 0$, then $u(L)\geq m-1$.

\item If $\Delta_L\ne 0$ and $u(L)=m-1$, then
\[ \Delta_L= p\, \ol{p}\cdot n\]
for some $p\in \L$ and some negligible $n$.
\en
\end{corollary}

\begin{proof}
We denote by $J$ the unlink with $m$-components. It follows from Lemma \ref{lem:alexsplit} that $\b(J)=m-1$ and
$\tordelta_J\doteq 1$.
The first  statement of the corollary follows immediately from the first statement of Theorem \ref{mainthm} together with Lemma \ref{lem:rankmminus1}.

Now suppose that $u(L)=m-1$ and $\Delta_L\ne 0$.
In this case $\b(L)=0$ and $\deltator_L=\Delta_L$. It thus follows that $|\b(L)-\b(J)|=m-1=u(L)=g(L,J)$.
The desired statement follows immediately from Theorem \ref{lem:torsionpolynomials} and
$\tordelta_J\doteq 1$.
\end{proof}

We also have the following corollary which significantly strengthens~\cite[Theorem~4.2]{CFP13}.

\begin{corollary}\label{cor:split}
Let $L$ be an $m$-component link. Then the following hold:
\bn
\item \label{item:corsplit-one} We have $m-1-\b(L)\leq \sp(L)$. In particular if $\Delta_L\ne 0$, then $\sp(L)\geq m-1$.
\item \label{item:corsplit-two} If $\Delta_L\ne 0$ and $\sp(L)=m-1$,   then
\[ \Delta_L=  \prod_{i=1}^m \Delta_{L_i}(t_i)\,\,\,\cdot \,\,\, p\, \ol{p}\cdot n\]
for some $p\in \L$ and some negligible $n$.
\en
\end{corollary}

The corollary is deduced from Theorem~\ref{lem:torsionpolynomials} in almost the same way as Corollary  \ref{cor:unlink}, except that we now apply Lemma \ref{lem:alexsplit} to the split link whose components are precisely the components of $L$ when they are considered as individual knots.

Finally the following corollary is also proved in the same way as Corollary~\ref{cor:unlink}, except here the knot types occurring in some putative split link, obtained by $m-1$ crossing changes on $L$, are unknown. We leave the details to the reader.

\begin{corollary}\label{cor:wsp}
Let $L$ be an $m$-component link. Then the following hold:
\bn
\item We have $m-1-\b(L)\leq \wsp(L)$. In particular if $\Delta_L\ne 0$, then $\wsp(L)\geq m-1$.
\item If $\Delta_L\ne 0$ and $\wsp(L)=m-1$,   then
\[ \Delta_L= \prod_{i=1}^m p_i(t_i)\,\,\,\cdot \,\,\, p\,\ol{p}\cdot n\]
for some $p_i(t_i)\in \Z[t_i^{\pm 1}]$, $i=1,\dots,m$, some $p \in \L$ and some negligible $n$.
\en
\end{corollary}

The reader may compare Corollary~\ref{cor:wsp} to \cite[Corollary~4.1]{Ka13}.

\section{Examples of unlinking and splitting number computations}\label{section:examples-unlinking-splitting}

\subsection{Unlinking numbers}\label{section:application-unlinking-numbers}
Kohn \cite{Koh93} considered the unlinking numbers of 2-component links with 9 or fewer crossings.  For most 3-component links with 9 or fewer crossings, the deduction of the unlinking number follows easily from elementary considerations of linking numbers, unknotting numbers of components, and certain sublinks being nontrivial.  In this section we show that Alexander modules enable a quick calculation of the unlinking numbers of the remaining five 3-component links with 9 or fewer crossings.  These five links are $L6a4$, $L8a16$, $L9a46$, $L9a53$ and $L9a54$.  We remark that the conclusions of this subsection already follow from~\cite{Ka13}, so we will be brief.

\begin{itemize}
\item The 3-component link $L8a16$ has unknotted components and Alexander polynomial
\[ (t_1-1)(t_2-1)(t_3-1)(t_2t_3-1).\]
Since $t_2t_3-1$ is not a norm it follows from Corollary \ref{cor:unlink} that the unlinking number is at least three.  In fact the unlinking number is equal to three.

\item We now consider the 3-component link $L9a54$, which has unknotted components. Its Alexander polynomial is
\[ (t_3 - 1)(t_2 - 1)(t_1 - 1)(t_3^2  - t_3 + 1).\]
Again, since $t_3^2  - t_3 + 1$ is not a norm it follows from Corollary \ref{cor:unlink} that the unlinking number is at least three.  In fact the unlinking number is equal to three.

\item The 3-component links $L6a4$, $L9a46$ and $L9a53$ have nonzero Alexander polynomial, hence unlinking numbers at least two
by Corollary \ref{cor:unlink}.  In fact the unlinking numbers of these links are equal to two.

\item We also briefly consider one 2-component link, the link $L9a1$. It has two unknotted components, and its Alexander polynomial is
\[ (t_2 - 1)(t_1 - 1)(2t_2^2  - 3t_2 + 2).\]
So it follows from Corollary \ref{cor:unlink} that the unlinking number is at least two.
In fact the unlinking number is equal to two.  This was already shown by Kohn~\cite{Koh93} using other methods.
\end{itemize}

\subsection{Band-claspings of split links}\label{section:hopfbandlinking}

Let $K\sqcup J$ be a 2-component split link. Pick an embedding $f\co D=D^2\to S^3$ such that $f(D)\cap K=f(\partial D)\cap K$ is an interval and such that $f(D)$ intersects $J$ transversally in one point in the interior of $f(D)$. Then we write
\[ K'=K\sm f(\partial D) \cup \overline{f(\partial D)\sm K} \]
and we refer to $K'\cup J$ as a \emph{band-clasping of $K$ and $J$}.  See Figure~\ref{fig:band-clasping}.

\begin{figure}[h]
 \begin{center}
  \labellist\small
    \pinlabel{$J$} at 235 198
    \pinlabel{$K$} at 7 198
    \pinlabel{$D$} at 100 160
   \pinlabel{$J$} at 565 198
    \pinlabel{$K'$} at 344 198
    \endlabellist
\includegraphics[scale=0.5]{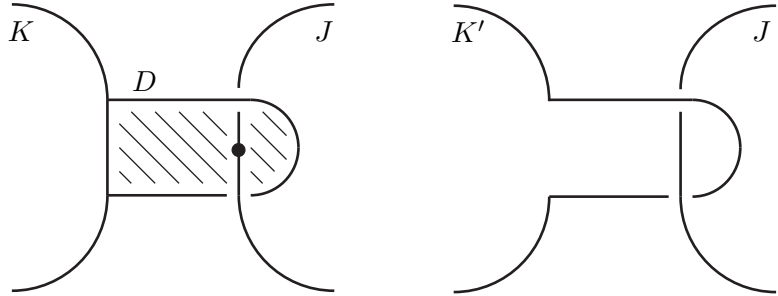}
\caption{Band-clasping.}
\label{fig:band-clasping}
\end{center}
\end{figure}

In Figure \ref{fig:split1} we show a band-clasping of two trefoils.
If we can find a projection onto a plane such that the projections of $K$ and $f(D)$ intersect only in the projection of $K\cap f(D)$, then we say that  $K'\cup J$ is the \emph{trivial band-clasping of $K$ and $J$}. It is straightforward to see that in that case the resulting link does not depend on the choice of~$f$.

\begin{figure}[h]
 \begin{center}
  \labellist\small
    \pinlabel{$J$} at 185 67
    \pinlabel{$K$} at 120 83
    \endlabellist
\includegraphics[scale=1]{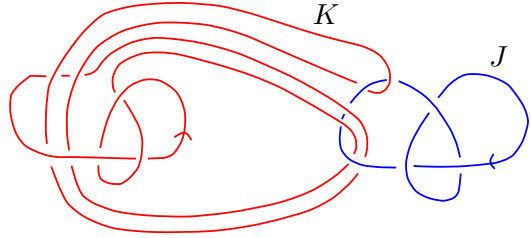}
\caption{A band-clasping of two trefoils.}
\label{fig:split1}
\end{center}
\end{figure}

We have the following observation about Alexander polynomials of band-claspings.

\begin{proposition}\label{lem:alexhopflinkings}
Let $L$ be a band-clasping of $K$ and $J$, then
\[ \Delta_L(s,t)=\Delta_K(t)\cdot \Delta_J(s)\cdot g\,\ol{g}\]
for some non-zero $g\in \L$. Furthermore $g=1$ if the band-clasping is trivial.
\end{proposition}

For example, using Kodama's program \emph{knotGTK} we can show that for the link $L$ in Figure \ref{fig:split1} we have
\[ \Delta_L(s,t)=(1-s+s^2)(1-t+t^2)(s^{-1}-1+t)(s-1+t^{-1}).\]
Note that  band-claspings have splitting number $1$. The lemma is thus a consequence of Corollary \ref{cor:split}, but we prefer to give a sketch of a proof which is particular to this class of links.

\begin{proof}[Sketched proof of Proposition~\ref{lem:alexhopflinkings}]
First of all, it is well-known, and can be shown using a Mayer--Vietoris argument, that the Alexander polynomial of the trivial band-clasping of $K$ and $J$ equals $\Delta_K(t)\cdot \Delta_J(s)$. Furthermore the proof of \cite[Theorem~1.1]{Mi98} carries over to show that  any  band-clasping $L$ of $K$ and $J$ is in fact ribbon concordant
 to the trivial band-clasping of $K$ and $J$. (We refer to \cite{Tr69} or alternatively \cite[p.~189]{Sav02} for the definition of ribbon concordance.) It then follows from standard arguments, e.g.\ by a variation on~\cite[Theorem~B]{Ka78}, that
\[ \Delta_L(s,t)=\Delta_K(t)\cdot \Delta_J(s)\cdot g\,\ol{g}\]
for some non-zero $g\in \L$.
\end{proof}

It can be shown by an argument completely analogous to that of~\cite[Theorem~1]{Kon79}, that any 2-component link with  splitting number $1$ is a  band-clasping of its components.
Moreover it seems likely, but we will not provide a proof, that  in Proposition \ref{lem:alexhopflinkings}  any non-zero $g$ can be realized by a band-clasping.
If this is correct, then this will  in particular show, except for determining the negligible factor precisely, that the conclusion of Corollary \ref{cor:split} (\ref{item:corsplit-two}) is optimal.

\subsection{Splitting numbers}\label{sec:splitnumbers}
In an earlier paper~\cite{CFP13}, two of us together with Jae Choon Cha already discussed splitting numbers in detail. In this section we will revisit some of the results from that paper.

First we remind the reader that in the calculation of the splitting number, one only allows crossing changes between different components. It is straightforward to show (see \cite[Lemma~2.1]{CFP13}) that the splitting number has the same parity as the sum of all linking numbers $\lk(L_i,L_j)$ with $i> j$. For example, if $L$ is a 2-component link with odd linking number, then the splitting number is also necessarily odd.

In \cite{CFP13} Alexander polynomial techniques were used to derive splitting number conclusions for 2-component linking number one links with at least one knotted component.  When both components were unknotted, covering link calculus was used, in which one studies the preimage of one component of the link in the covering space branched along the other component; see \cite{CK08,Cha09,CO93} for more on covering link calculus.
Some, but not all, of the conclusions obtained in \cite{CFP13} using covering links can be drawn using Corollary~\ref{cor:split}.
For example, in~\cite{CFP13} we investigated the link $L12n1320$, shown in Figure~\ref{fig:L12n1320}.

\begin{figure}
\begin{center}
\includegraphics[scale=0.4]{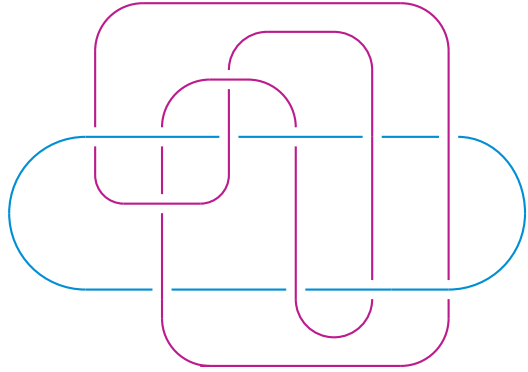}
\caption{The link $L12n1320$.}\label{fig:L12n1320}
\end{center}
\end{figure}

This is a 2-component link with linking number one and unknotted components.  It was shown in \cite[Section~5.2]{CFP13} that the splitting number is 3.  According to \emph{knotGTK}~\cite{Kod}, the Alexander polynomial is:
\[t_1^3t_2^3-2t_1^2t_2^3-t_1^3t_2^2+t_1t_2^3+5t_1^2t_2^2-4t_1t_2^2-4t_1^2t_2+5t_1t_2+t_1^2-t_2-2t_1+1,\]
which factors as
\[(t_1-1)(t_2-1)(t_1^2t_2^2-t_1t_2^2+3t_1t_2-t_1+1).\]
Since the last factor is not a norm, Corollary~\ref{cor:split} says that the splitting number is greater than 1.  In fact by the observation above, the splitting number of $L12n1320$ has to be odd, so it has to be at least 3. In fact  it is easy to verify that it is precisely~3.  The proof of this fact in \cite[Section~5.2]{CFP13} used twisted Alexander polynomials to show that a covering link is not slice, while \cite{BS13} used a Khovanov homology spectral sequence.

Similarly, the links  $L8a16$ and $L9a46$  were shown in \cite{CFP13} to have splitting number 3 using covering links.  They are 3-component links with nonzero Alexander polynomial, hence we also get from Corollary~\ref{cor:split} that the splitting number is at least $3$.

Note that for 2- and 3-component links this approach can only show that the splitting number is at least 3,
whereas the covering link techniques were sometimes sufficient to show that the splitting number is~5.

\subsection{Weak splitting numbers}

The 3-component link $L8a16$, shown in Figure~\ref{fig:l8a16}, has unknotted components and Alexander polynomial
\[ (t_1-1)(t_2-1)(t_3-1)(t_2t_3-1).\]

\begin{figure}
\begin{center}
\includegraphics[scale=0.5]{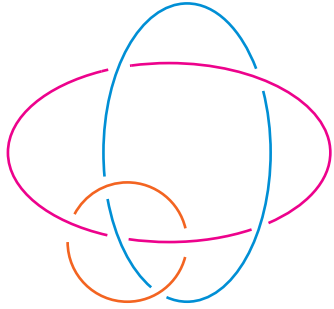}
\caption{The link $L8a16$.}\label{fig:l8a16}
\end{center}
\end{figure}

As well as having unlinking number~3 we see that $L8a16$ also has weak splitting number~3, by Corollary~\ref{cor:wsp}.
Similarly, we can also apply Corollary~\ref{cor:wsp} to prove that the link $L12n1320$ considered in Section~\ref{sec:splitnumbers} does not have weak splitting number~1.

\section{Knot types obtained from weak splitting operations}\label{section:knot-types-weak-splitting}

Recall the following notation from the introduction.
If a link $J$ can be obtained from a link $L$ by a sequence of $r$ crossing changes then we write $L \rightsquigarrow_r J$.
A sequence of crossing changes culminating in a split link is referred to as a \emph{splitting sequence}. Given knots $K_1,\dots,K_m$ we denote the split link whose components are these knots by $K_1 \sqcup \dots \sqcup K_m$.  Also we write $U$ for the unknot.

Given an $m$-component link $L$ with weak splitting number $\wsp(L)=r$, we investigate the question of which knot types can arise in a splitting sequence of length $r$.
 Theorem~\ref{thm:weak-split-no-blanchfield-r-plus-one-component} below concerns the case $r=m-1$.

\begin{theorem}\label{thm:weak-split-no-blanchfield-r-plus-one-component}
  Let $L$ be an $m$ component link with $\Delta_L \neq 0$ and $\wsp(L) = m-1$.
  Then for any two splitting sequences  $L \rightsquigarrow_{m-1} K_1 \sqcup \dots \sqcup K_{m}$ and  $L \rightsquigarrow_{m-1} J_1\sqcup \dots \sqcup J_{m}$ we  have
  \[\bigoplus_{i=1}^{m}\Bl_{K_i}(t_i) \sim \bigoplus_{i=1}^{m}\Bl_{J_i}(t_i),\]
where $\sim$ indicates equivalence in the Witt group of linking forms. In particular
\[\prod_{i=1}^{m} \Delta_{K_i}(t_i) \cdot  f \, \ol{f} = \prod_{i=1}^{m} \Delta_{J_i}(t_i) \cdot g \, \ol{g} \cdot n\]
for some non-zero polynomials $f,g$ and some negligible  $n\in \L$.
\end{theorem}

\begin{proof}
  In this proof write $\K:= K_1 \sqcup \dots \sqcup K_{m}$ and  $\J := J_1\sqcup \dots \sqcup J_{m}$.  Since $\Delta_L \neq 0$ we have $\beta(L) =0$, while $$\beta(\K) = \beta(\J) = r$$ by Lemma~\ref{lem:alexsplit}.  By Theorem~\ref{mainthm} we have that $\Bl_{\K} \oplus -\Bl_L$ and $\Bl_{\J} \oplus - \Bl_L$ are metabolic and therefore both are zero in the Witt group.   In particular they are equivalent in the Witt group, from which it follows that $\Bl_{\K}=\Bl_{\J}$ in the Witt group.  By Lemma~\ref{lem:blanchfield-split} the Blanchfield forms of $\K$ and $\J$ are the Witt sums of the Blanchfield forms of their constituent knots.

The second statement is now a consequence of Lemma \ref{lem:sameorder} and Lemma \ref{lem:ordernorm}.
\end{proof}

Adams \cite{Ad96} gave the first example of a 2-component link $L$ with unknotted components and weak splitting number one, such that any crossing change which splits $L$ necessarily turns one of the two components of $L$ into a nontrivial knot.
In the final paragraph of \cite{Ad96}  Adams asked (see Question \ref{question:adams}) whether there exist such examples, where in addition we may guarantee high complexity of a component arising from a single splitting crossing change.
The following theorem gives an affirmative answer to Adams' question.\\

\noindent \textbf{Theorem \ref{thm:high-crossing-no-knots-from-wsp-one-links}.}\emph{
Fix $c \in \mathbb{N}$.  There exists a 2-component link $L$ with unknotted components such that such for any knot $K$ with $L \rightsquigarrow_1 K \sqcup U$, the crossing number of $K$ is at least~$c$.}\\

The proof of Theorem \ref{thm:high-crossing-no-knots-from-wsp-one-links} will require the remainder of this section.
The examples we construct are inspired by the construction of Adams \cite{Ad96}, but we remark that we have to change the links from \cite{Ad96} slightly, since the links in \cite[Figure~4]{Ad96} are boundary links and therefore have $\Delta_L=0$ and $\beta(L)=1$, whereas we require $\Delta_L \neq 0$ and $\beta(L)=0$ in order to apply our results.
 \begin{figure}[h]
    \labellist\small
    \pinlabel{$T$} at 60 90
    \pinlabel{$T$} at 360 90
    \endlabellist
  \includegraphics[scale=0.5]{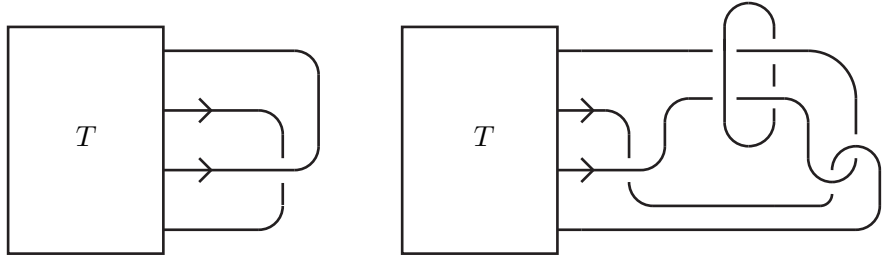}
    \caption{Left: the box denotes a tangle $T$ such that the diagram is an unknotting number one knot with the unknotting crossing isolated.  Right: replace the strands outside the box as shown to get a 2-component link $L_T$ with unknotted components and weak splitting number one.}
    \label{figure:wsp-one-links}
      \end{figure}

Choose $n$ to be such that $2n+1 \geq c$.
Choose an irreducible Laurent polynomial $\Delta(t) = a_0(1+t^{2n}) + a_1(t+t^{2n-1}) + \dots + a_{n-1}(t^{n-1} + t^{n+1}) + a_nt^n$ with $\Delta(1)=1$ and degree $2n$, where $2n=p-1$ for $p$ an odd prime.  For example choosing $n$ so that $2n=p-1$ for $p$ an odd prime greater than or equal to $c$, and taking $a_{2i}=1$ and $a_{2i+1}=-1$ for $i=0,\dots,\lfloor n/2 \rfloor$ gives rise to such a polynomial, since this is a cyclotomic polynomial and cyclotomic polynomials are irreducible.

According to the main theorem of~\cite{Kon79}, there exists
 an unknotting number one knot~$J$ with $\Delta_{J}(t) \doteq \Delta(t)$.  Let~$T$ be a tangle such that the picture on the left hand side of Figure~\ref{figure:wsp-one-links} is a diagram for $J$, where we isolated  a crossing, at which a crossing change results in an unknot.
  If necessary, switch~$J$ for one of either its reverse $rJ$, its mirror image $\ol{J}$ or $r\ol{J}$, in order to arrange that the orientations are as shown on the left of Figure~\ref{figure:wsp-one-links}.  (These orientations will soon be important for simplifying the construction of a Seifert surface.)
Replace the strands outside the box with the arrangement on the right hand side of Figure~\ref{figure:wsp-one-links}, to obtain a 2-component link with unknotted components which we call $L_T$.  Changing one crossing of $L_T$, in the clasp on the right, yields $J \sqcup U$.  This construction is an adaptation of that of \cite[Figure~4]{Ad96}.

\begin{lemma}\label{lem:non-zero-alex-poly}
The links $L_T$ constructed above have nonzero Alexander polynomial $\Delta_{L_T} \neq 0$.
\end{lemma}

Before giving the proof we recall the definition of the Sato-Levine invariant of a 2-component link $L=L_1\cup L_2$ with linking number zero~\cite{Sat84}. Pick  two Seifert surfaces~$F_1$ and~$F_2$ in $S^3$, with $\partial F_i = L_i$, $F_1 \cap L_2 = F_2 \cap L_1 = \emptyset$ and $F_1\pitchfork F_2$.  The intersection $F_1 \cap F_2$ is a link $J \subset S^3$.  Choose an orientation of $J$, a framing for the normal bundle of $J$ in $F_1$ and a framing for the normal bundle of $J$ in $F_2$, such that the first two agree with the orientation of $F_1$, the first and the third agree with the orientation of $F_2$, and all three agree with the orientation of $S^3$.  Together the framings of the normal bundles to $J$ in $F_1$ and $F_2$ give a framing for the normal bundle of $J$ in $S^3$.  The framed bordism class of the link $J$ then defines the Sato-Levine invariant. Recall that two framed links in $S^3$ are framed bordant if and only if the sums of their framing coefficients are equal, since we can use the Pontryagin-Thom construction to produce an element of $\pi_3(S^2) \cong \Z$, with the Hopf invariant yielding the isomorphism to~$\Z$.

\begin{proof}[Proof of Lemma~\ref{lem:non-zero-alex-poly}]
We start with the following claim.
  \begin{claim}
    The links $L_T$ above have Sato-Levine invariant $-1$.
  \end{claim}
To prove the claim, apply the Seifert algorithm to the left hand component of $L_T$, on the right of Figure~\ref{figure:wsp-one-links}.  Call this component $L_1$ and the resulting Seifert surface~$F_1$.  Construct a Seifert surface~$F_2$ for the other component~$L_2$ by taking the obvious disc and tubing along $L_1$ where $L_1$ hits the disc, passing the tube around the clasp.  This makes Seifert surfaces $F_1,F_2$ for $L_1,L_2$ respectively with $F_1 \cap L_2 = \emptyset = F_2 \cap L_1$.  The orientation is important for ensuring that the Seifert algorithm gives a surface $F_1$ disjoint from~$L_2$.  The intersection $F_1 \cap F_2$ is a single circle and the self linking of $F_1 \cap F_2$ from the framing induced by the Seifert surfaces is $-1$; it can be seen that a full negative twist in the induced framing arises when passing around the clasp.  This completes the proof of the claim.

As was shown in~\cite[Theorem~4.1]{Co85}, the Sato-Levine invariant of a 2-component link~$L$ with linking number zero is equal to minus the coefficient of $z^3$ in the Conway polynomial $\nabla_L(z)$.  Thus the Conway polynomial is nonzero.

According to Kawauchi~\cite[Proposition~7.3.14]{Ka96} we may relate the multivariable and single variable Alexander polynomials by:
\[\Delta_L(t,t) (t-1) = \Delta_L(t).\]
 Thus, to show that the multivariable Alexander polynomial is nonzero it suffices to show that $\Delta_L(t) \neq 0$.  Suppose that $V$ is an $m \times m$ Seifert matrix for $L$ arising from a connected Seifert surface.  Then $$t^{-m/2}\Delta_L(t) \doteq \det(t^{1/2}V-t^{-1/2}V^T) = \nabla_L(t^{1/2} - t^{-1/2}) = \nabla_L(z);$$ the change of variables is $z = t^{1/2} - t^{-1/2}$.  Thus if $\Delta_L(t)=0$ then $\nabla_L(z)=0$.  The fact shown above that $\nabla_L(z) \neq 0$ therefore completes the proof of Lemma~\ref{lem:non-zero-alex-poly}.
\end{proof}

The next result follows immediately from Alexander's original definition; compare also \cite[Exercise 8.C.12, page 208]{Ro76}. The proof
is left to the reader.
For a Laurent polynomial $p(t) = \sum_{i \in \Z} a_i t^i \in \Z[t^{\pm 1}]$ we define $\deg(p(t))$ to be the difference
$\deg(p(t)) := \max \{j \in \Z \, | \, a_j \neq 0\} - \min \{k \in \Z \, | \, a_k \neq 0\}$.

\begin{lemma}\label{lem:lower-bound-crossing-number}
Let $K$ be a nontrivial knot and $c$ be its crossing number. Then the degree of the Alexander polynomial satisfies $\deg\Delta_K \le c-1$.
\end{lemma}

%
%
%
%

\begin{proof}[Continuation of the proof of Theorem~\ref{thm:high-crossing-no-knots-from-wsp-one-links}]
Consider the links $L_T$ constructed above.  We have $L \rightsquigarrow_1 J \sqcup U$, where $\deg(\Delta_J) = 2n$ and $2n+1 \geq c$; recall that $n$ was chosen to satisfy this property with respect to $c$.  For any knot $K$ with Alexander polynomial having degree $2n$ we have $2n \leq k-1$, where $k$ is the crossing number of $K$.  Thus we have $c \leq 2n+1 \leq k$.  It therefore suffices to show that any knot $K$ arising from one splitting crossing change on $L_T$ has Alexander polynomial containing $\Delta_J(t)$ as a factor.

Since $\Delta_{L_T} \neq 0$, we have that $\beta(L) = 0$, whereas $\beta(J \sqcup U) =1$ by Lemma~\ref{lem:alexsplit}.  Therefore by Theorem~\ref{lem:torsionpolynomials} and another application of Lemma~\ref{lem:alexsplit} we have that
\[\Delta_{L_T}(t_1,t_2) = \deltator_{J \sqcup U} \cdot f \, \ol{f} \cdot m = \Delta_J(t_1) \cdot f \, \ol{f} \cdot m\]
for some $f \in \L$ and some negligible  $n\in \L$.

Now suppose that we have some other splitting crossing change on $L$ yielding $K \sqcup U$.  Then similarly to above we have
\[\Delta_{L_T}(t_1,t_2) = \deltator_{K \sqcup U} \cdot g \,  \ol{g} \cdot m' = \Delta_K(t_1) \cdot g \, \ol{g} \cdot m'\]
for some $g \in L$ and some negligible  $m'\in \L$.
Therefore
\begin{equation}\label{eq:exponents}
\Delta_J(t_1) \cdot f \,\ol{f} \cdot m = \Delta_K(t_1) \cdot g \,\ol{g} \cdot m'.
\end{equation}
The ring $\L$
is a UFD and $\Delta_J$ is irreducible. Therefore a non--negative number $a_{\Delta_K}$ such that $\Delta_J^{a_{\Delta_K}}$ divides
$\Delta_K$, but $\Delta_J^{1+a_{\Delta_K}}$ does not, is well--defined. Similarly we define $a_{f}$, $a_{\ol{f}}$, $a_{g}$, $a_{\ol{g}}$.
As $\Delta_J$ is symmetric, we infer that $a_{f}=a_{\ol{f}}$ and $a_g=a_{\ol{g}}$. Notice that $\Delta_J$, being a non--trivial knot polynomial, does not divide negligible polynomials $m$ and $m'$.

The maximal exponent $a$ such that $\Delta_J$ divides the left hand side of \eqref{eq:exponents}
is $1+2a_f$. For the right hand side it is $a_{\Delta_K}+2a_g$. This implies
that $a_{\Delta_K}$ is odd, in particular $\Delta_J$ divides $\Delta_K$. This shows that $\deg\Delta_K\ge\deg\Delta_J=2n$.
Recall that $n$ was chosen so that $2n+1 \geq c$, and by Lemma~\ref{lem:lower-bound-crossing-number} this implies that the crossing number of~$K$ is at least $c$.
\end{proof}

\section{Questions}\label{section:conjectures}

Kohn~\cite{Koh93} initiated the study of unlinking numbers for links with more than one component.
There are five 2-component 9 crossing links for which Kohn could not compute the unlinking number, namely
\[ 9^2_3=L9a30, 9^2_{15}=L9a15, 9^2_{27}=L9a17 ,9^2_{31} = L9a2 \mbox{ and }9^2_{36} = L9a10,\]
where the names come from Rolfsen's book~\cite{Ro76} and the Linkinfo tables~\cite{CL} respectively.
For each link the question is whether the unlinking number is two or three.  Kanenobu recently announced a proof that the unlinking number of $L9a30$ is 3.
Unfortunately the  techniques of this paper do not help.  It would be very interesting if it could be shown that one of the four remaining links has unlinking number 3.

\end{document}